\documentclass[12pt,a4paper,reqno]{article}
\usepackage{stix}
\usepackage{amsmath,amsfonts}
\usepackage{amsthm}
\usepackage{upref}
\usepackage{enumerate}
\usepackage{authblk}
\usepackage[margin=3.3cm]{geometry}
\usepackage[colorlinks=true]{hyperref}

\title{Towards a perturbation theory for eventually positive semigroups}

\author[1]{Daniel Daners}%
\author[2]{Jochen Gl\"uck\thanks{Partially supported by a scholarship within the
    scope of the LGFG Baden-W\"urttemberg, Germany.}}%
\affil[1]{School of Mathematics and Statistics, University of Sydney,
  NSW 2006, Australia\authorcr%
  \nolinkurl{daniel.daners@sydney.edu.au}}%
\affil[2]{Institut f\"ur Angewandte Analysis, Universit\"at Ulm,
  D-89069 Ulm, Germany\authorcr%
  \nolinkurl{jochen.glueck@uni-ulm.de}}%

\date{\today}

\makeatletter
\hypersetup{pdfauthor={Daniel Daners, Jochen Glueck}}
\hypersetup{pdftitle={\@title}}
\makeatother

\newtheorem{theorem}{Theorem}[section]
\newtheorem{lemma}[theorem]{Lemma}
\newtheorem{proposition}[theorem]{Proposition}
\newtheorem{corollary}[theorem]{Corollary}

\theoremstyle{definition}

\newtheorem*{definition_no_number}{Definition}

\newtheorem{example}[theorem]{Example}
\newtheorem{examples}[theorem]{Examples}

\theoremstyle{remark}
\newtheorem{remark}[theorem]{Remark}

\numberwithin{equation}{section}

\DeclareMathOperator{\id}{\mathit I}
\DeclareMathOperator{\one}{\mathbf 1}
\DeclareMathOperator{\im}{im}
\DeclareMathOperator{\spec}{\sigma}
\DeclareMathOperator{\resSet}{\rho}
\DeclareMathOperator{\Res}{\mathcal{R}}
\DeclareMathOperator{\spr}{r}
\DeclareMathOperator{\spb}{s}
\DeclareMathOperator{\repart}{Re}

\DeclareMathOperator{\Cont}{\mathcal{C}}

\newcommand{\bbN}{\mathbb{N}}

\newcommand{\bbR}{\mathbb{R}}

\newcommand{\bbC}{\mathbb{C}}

\newcommand{\calL}{\mathcal{L}}
\newcommand{\calT}{\mathcal{T}}

\newcommand{\phdot}{\mathord{\,\cdot\,}}

\let\oldthebibliography\thebibliography
\renewcommand\thebibliography[1]{
  \oldthebibliography{#1}
  \setlength{\parskip}{0pt}
  \setlength{\itemsep}{0pt plus 0.3ex}
  \small
}

\begin{document}
\maketitle

\begin{abstract}
  We consider eventually positive operator semigroups and study the
  question whether their eventual positivity is preserved by bounded
  perturbations of the generator or not.  We demonstrate that eventual
  positivity is not stable with respect to large positive perturbation
  and that certain versions of eventual positivity react quite
  sensitively to small positive perturbations. In particular we show
  that if eventual positivity is preserved under arbitrary positive
  perturbations of the generator, then the semigroup is positive. We
  then provide sufficient conditions for a positive perturbation to
  preserve the eventual positivity. Some of these theorems are
  qualitative in nature while others are quantitative with explicit
  bounds.
\end{abstract}

{ 
  \renewcommand{\thefootnote}{}%
}

\section{Introduction}
\label{section:introduction-preliminaries}

For positive $C_0$-semigroups, it is easy to derive basic perturbation
results. If, for instance, $A$ generates a positive $C_0$-semigroup on a
Banach lattice $E$, $B$ is a positive operator and $M$ is a
multiplication operator on $E$ (see \cite[Section~C-I-9]{Arendt1986}),
then it is not difficult to show that the semigroup generated by $A + B
+ M$ is also positive. In the present paper we study the problem whether
such a perturbation result is still true for \emph{eventually} positive
semigroups.

An \emph{eventually positive semigroup} is a $C_0$-semigroup $(e^{tA})$
on, say, a complex Banach lattice $E$ such that, for every initial value
$0 \leq f \in E$, the trajectory $e^{tA}f$ becomes positive for large
enough $t$. Motivated by applications to partial differential equations
(see e.g.\ \cite{Ferrero2008, Gazzola2008, Daners2014}; see also 
\cite{Sweers2016a} for an overview over related elliptic problems) 
and by the rapid development
of a corresponding theory in finite dimensions (see for instance
\cite{Noutsos2008, Olesky2009, Ellison2009, Erickson2015}), 
a study of
eventually positive semigroups on Banach lattices was initiated in a
series of recent papers \cite{Daners2016, Daners2016a, DanersNOTE}. In
particular, these papers clarified that there are several distinct
notions of eventual positivity such as an \emph{individual} and a
\emph{uniform} one which are worthwhile studying. For the convenience of the
reader we recall the exact definitions of these notions at the end of
the introduction as we are going to need them throughout the paper.

We shall see in this article that perturbation theory is much more
subtle for eventually positive semigroups than it is for positive
semigroups. We first demonstrate by a number of counterexamples in
Section~\ref{section:counterexamples} what is \emph{not} true. In
particular we will see that, in sharp contrast to the case of positive
semigroups, eventual positivity of a semigroup is in general lost if we
perturb its generator by a positive operator of large norm; this is 
related to a recent result of Shakeri and Alizadeh for perturbations 
of eventually positive matrices \cite[Proposition~3.6]{Shakeri2017}. Moreover,
one of our examples shows that \emph{individual} eventual positivity is
not even stable with respect to small positive perturbations. This is
the reason why we focus on \emph{uniform} eventual positivity throughout
the rest of the paper. In Section~\ref{section:resolvents} we prove
qualitative as well as quantitative perturbation results for eventually
positive resolvents of operators, and in
Section~\ref{section:semigroups} we prove qualitative and quantitative
perturbation results for $C_0$-semigroups. In the appendix we consider
rank-$1$-perturbations of linear operators and prove explicit formulas
for their resolvents and for the semigroups generated by those
operators; these formulas are needed in the main text.

It is important to note that our results are far from constituting a
complete perturbation theory for eventually positive semigroups; in
fact, we leave much more questions open than we solve. It is our hope
though that, by exposing some surprising phenomena, the present article
can serve a starting point for further research on the topic.

\paragraph*{Preliminaries}

Throughout, we use the notation and the terminology from
\cite{Daners2016, Daners2016a, DanersNOTE}. For the convenience of the
reader we recall what we need throughout the paper. We assume
familiarity with the theory of real and complex Banach lattices (see for
instance \cite{Schaefer1974, Meyer-Nieberg1991} for standard references
on this topic) and with the basic theory of $C_0$-semigroups (see for
instance \cite{Pazy1983, Engel2000, Engel2006}).

For every $\lambda \in \bbC$ and every real number $r > 0$ we denote by
$B(\lambda,r) := \{z \in \bbC\colon |z-\lambda| < r\}$ the open ball in
$\bbC$ of radius $r$.

If $E,F$ are real or complex Banach spaces, then we denote the space of
all bounded linear operators from $E$ to $F$ by $\calL(E;F)$ and we
abbreviate $\calL(E) := \calL(E;E)$. The \emph{identity operator} on $E$
is denoted by $\id_E \in \calL(E)$. For every $T \in \calL(E)$ the
\emph{spectral radius} of $T$ is denoted by $\spr(T)$. For
every densely defined linear operator $A\colon E \supseteq D(A) \to F$
we denote by $A'\colon F \supseteq D(A') \to E'$ the \emph{dual
operator} of $A$, where $E'$ and $F'$ are the \emph{dual spaces} of $E$ 
and $F$. For all $y \in F$
and all $\varphi \in E'$ we define $y \otimes \varphi \in \calL(E;F)$ by
$(y \otimes \varphi)z := \langle \varphi, z \rangle y$ for all $z \in
E$. It is easy to see that the operator norm of $y \otimes \varphi$ is
given by $\|y \otimes \varphi\| = \|y\| \|\varphi\|$. Recall that every
rank-$1$-operator in $\calL(E;F)$ is of the form $y \otimes \varphi$ for
appropriate vectors $y \in F \setminus \{0\}$ and $\varphi \in E'
\setminus \{0\}$.

Given a linear operator $A\colon E \supseteq D(A) \to E$ on a complex
Banach space $E$ we denote its \emph{spectrum} and \emph{resolvent set}
by $\spec(A)$ and $\resSet(A) := \bbC \setminus \spec(A)$,
respectively. Note that if $\resSet(A)\neq\emptyset$, then $A$ is
necessarily closed.  The \emph{spectral bound} of $A$ is given by
\begin{displaymath}
  \spb(A) := \sup \{\repart \lambda\colon \lambda \in \spec(A)\} \in
  [-\infty,\infty],
\end{displaymath}
We define the \emph{resolvent} of $A$ at $\lambda\in\resSet(A)$ by
$\Res(\lambda,A) := (\lambda I-A)^{-1}$. The resolvent
$\Res(\phdot,A)\colon \resSet(A) \to \calL(E)$ is an analytic map.  Any
pole of this analytic map is an isolated point of $\spec(A)$ and in fact
an eigenvalue of $A$; see \cite[Theorem~2 in
Section~VIII.8]{Yosida1995}. Let $\lambda_0$ be a pole of the resolvent
of $A$. We call $\lambda_0$ a \emph{geometrically simple eigenvalue} of
$A$ if the eigenspace $\ker(\lambda_0 I - A)$ is one-dimensional; we
call $\lambda_0$ an \emph{algebraically simple eigenvalue} of $A$ if the
spectral projection $P$ associated with $\lambda_0$ has one-dimensional
range. The eigenvalue $\lambda_0$ is algebraically simple if and only if
it is geometrically simple and $\im(P)=\ker(\lambda_0 I - A)$.  Also, if
$\lambda_0$ is algebraically simple, then $\lambda_0$ is a simple pole
of $\Res(\phdot,A)$. Let $A\colon E \supseteq D(A) \to E$ be a linear
operator with non-empty resolvent set on a complex Banach space $E$.  An
operator $K \in \calL(E)$ is called \emph{$A$-compact} if there is a
$\lambda_0 \in \rho(A)$ such that $K \Res(\lambda_0,A)$ is compact. By
the resolvent equation this is equivalent to $K \Res(\lambda,A)$ being
compact for every $\lambda \in \rho(A)$. Note that every compact
operator $K \in \calL(E)$ is naturally $A$-compact. Moreover, if $A$ has
compact resolvent, then every operator $K \in \calL(E)$ is $A$-compact.

A complex Banach lattice $E$ is by definition the complexification of a
real Banach lattice $E_\bbR$ which we call the \emph{real part} of
$E$. The \emph{positive cone} of a real or complex Banach lattice $E$ is
denoted by $E_+$.  A vector $f \in E$ is called \emph{positive}, which
we denote by $f \geq 0$, if $f \in E_+$.  For two elements $f,g\in E$ in
case of a real Banach lattice or $f,g\in E_\bbR$ in case of a complex
Banach lattice we write, as usual, $f \leq g$ if $g-f \geq 0$.  We write
$f < g$ if $f \leq g$ but $f \neq g$. The dual space $E'$ of a real or
complex Banach lattice $E$ is again a real or complex Banach lattice,
where a functional $\varphi \in E'$ fulfils $\varphi \geq 0$ if and only
if $\langle \varphi, x\rangle \geq 0$ for all $x \in E_+$; we denote the
positive cone in $E'$ by $E'_+ := (E')_+$.

Let $E$ be a real or complex Banach lattice and let $u \in E_+$. The 
vector subspace
\begin{displaymath}
  E_u := \{x \in E\colon\text{there exists $c \geq 0$ with $|x| \leq cu$}\}
\end{displaymath}
of $E$ is called the \emph{principal ideal} generated by $u$.
We endow $E_u$ with the \emph{gauge norm} $\|\cdot\|_u$
with respect to $u$. The gauge norm is given by
\begin{displaymath}
  \|x\|_u := \inf \{c \geq 0\colon |x| \leq cu\}
\end{displaymath}
for all $x \in E_u$ and is at least as strong as the norm induced by
$E$, usually even stronger. Moreover it renders $E_u$ a (real or complex) 
Banach lattice. A
vector $u \in E$ is called a \emph{quasi-interior point of $E_+$} if $u
\geq 0$ and if $E_u$ is dense in $E$. If, for instance, $E$ is an
$L^p$-space over a $\sigma$-finite measure space $(\Omega,\mu)$ (where
$1 \leq p < \infty$) then $0 \leq u \in E$ is a quasi-interior point of
$E_+$ if and only if $u(\omega) > 0$ for almost all $\omega \in \Omega$.

Let $u \in E_+$. We call a vector $f \in E$ \emph{strongly positive with
  respect to $u$}, which we denote by $f \gg_u 0$, if there exists a
number $\varepsilon > 0$ such that $f \geq \varepsilon u$. This
condition is equivalent to the condition $f \geq 0$ and $u \in E_f$. An
operator $T \in \calL(E)$ is called \emph{strongly positive with respect
  to $u$}, which we denote by $T \gg_u 0$, if $Tf \gg_u 0$ for all $0 <
f \in E$.

Let $E$ be a complex Banach lattice. A linear operator $A\colon E
\supseteq D(A) \to E$ is called \emph{real} if $D(A) = E_\bbR \cap D(A)
+ i E_\bbR \cap D(A)$ and if $A$ maps $E_\bbR \cap D(A)$ to
$E_\bbR$. Clearly, an operator $T \in \calT(E)$ is real if and only if
$TE_\bbR \subseteq E_\bbR$. It is easy to see that a $C_0$-semigroup
$(e^{tA})_{t \geq 0}$ on $E$ is real if and only if $A$ is real.  A
linear operator $T \in \calL(E)$ on a real or complex Banach lattice $E$
is called \emph{positive} if $TE_+ \subseteq E_+$; we denote this by $T \geq
0$. In particular such an operator is real. A $C_0$-semigroup
$(e^{tA})_{t \geq 0}$ on $E$ generated by $A$ is called \emph{positive}
if $e^{tA} \geq 0$ for all $t \geq 0$. Furthermore, given $S,T \in
\calL(E)$ we write $S \leq T$ if $S$ and $T$ are both real operators and
if $T-S\geq 0$.

A real operator $T \in \calL(E)$ is called a \emph{multiplication
  operator} if there exists a number $c \geq 0$ such that $-c\id_E \leq
T \leq c\id_E$; it is also possible to define non-real multiplication
operators, but we have no need for this in the present article. All
multiplication operators on a Banach lattice constitute a vector space
which is usually called the \emph{center} of the Banach lattice; see for
instance \cite[Section~3.1]{Meyer-Nieberg1991} for more information.  We
recall how real multiplication operators can be characterised on two
important classes of complex Banach lattices, also explaining the name
``multiplication operator''.  Let $(\Omega,\mu)$ be a
$\sigma$-finite measure space and $K$ a compact Hausdorff space. Then
the real operator $T$ is a multiplicaton operator on $E=L^p(\Omega)$
with $1\leq p <\infty$ or $E=\Cont(K;\bbC)$ if and only if there exists
a function $h\in L^\infty(\Omega,\mu; \bbR)$ or $h \in \Cont(K;\bbR)$
respectively such that $Tf = hf$ for all $f \in E$.

\paragraph*{Notions of eventual positivity}

As in \cite{Daners2016, Daners2016a, DanersNOTE} we consider 
eventual positivity for resolvents of linear operators as well as for
$C_0$-semigroups. For the convenience of the reader we recall the
most important definitions now. First we recall several notions of
eventual positivity for resolvents:

\begin{definition_no_number}
  Let $A\colon E \supseteq D(A) \to E$ be a linear operator on a complex
  Banach lattice $E$ and let $\lambda_0 \in [-\infty,\infty)$ be either
  a spectral value of $A$ or $-\infty$.
  \begin{enumerate}
  \item[(a)] The resolvent of $A$ is called \emph{individually
      eventually positive at $\lambda_0$} if, for every $0 \leq f\in E$,
    there exists a real number $\lambda_1 > \lambda_0$ such that
    $(\lambda_0, \lambda_1] \subseteq \resSet(A)$ and such that
    $\Res(\lambda,A)f \geq 0$ for all $\lambda \in (\lambda_0,
    \lambda_1]$.
  \item[(b)] The resolvent of $A$ is called \emph{uniformly eventually
      positive at $\lambda_0$} if it is individually eventually positive
    at $\lambda_0$ and if the number $\lambda_1$ in (a) can be chosen to
    be independent of $f$.
  \end{enumerate}
  Now assume in addition that $u$ is a quasi-interior point of $E_+$.
  \begin{enumerate}
  \item[(c)] The resolvent of $A$ is called \emph{individually
      eventually strongly positive with respect to $u$ at $\lambda_0$}
    if, for every $0 < f\in E$, there exists a real number $\lambda_1 >
    \lambda_0$ such that $(\lambda_0, \lambda_1] \subseteq \resSet(A)$
    and such that $\Res(\lambda,A)f \gg_u 0$ for all $\lambda \in
    (\lambda_0, \lambda_1]$.
  \item[(d)] The resolvent of $A$ is called \emph{uniformly eventually
      strongly positive with respect to $u$ at $\lambda_0$} if it is
    individually eventually strongly positive with respect to $u$ at
    $\lambda_0$ and if the number $\lambda_1$ in (c) can be chosen to be
    independent of $f$.
  \end{enumerate}
\end{definition_no_number}
Note that one can also define various versions of \emph{eventual
  negativity} of a resolvent as was for instance done in
\cite[Definition~4.2]{Daners2016a}. We will, however, not discuss this
notion in detail here; it probably suffices to remark that all
perturbation results that we prove for eventually positive resolvents
have analogues for eventually negative resolvents (with similar proofs).

The most interesting case in the above definition is the case $\lambda_0
= \spb(A)$.  In fact, eventual positivity of the resolvent of $A$ at the
spectral bound is closely related to eventual positivity of the
semigroup (see for instance \cite[Theorem~1.1]{Daners2016a}.  Various
version of eventual positivity of a semigroup can be found in the
subsequent definition.
\begin{definition_no_number}
  Let $(e^{tA})_{t \geq 0}$ be a $C_0$-semigroup on a complex Banach
  lattice $E$.
  \begin{enumerate}
  \item[(a)] The semigroup is called \emph{individually eventually
      positive} if, for every $0 \leq f \in E$, there exists a time $t_0
    \geq 0$ such that $e^{tA}f \geq 0$ for all $t \geq t_0$.
  \item[(b)] The semigroup is called \emph{uniformly eventually
      positive} if it is individually eventually positive and if the
    time $t_0$ from (a) can be chosen to be independent of $f$.
  \end{enumerate}
  Now assume in addition that $u$ is a quasi-interior point of $E_+$.
  \begin{enumerate}
  \item[(c)] The semigroup is called \emph{individually eventually
      strongly positive with respect to $u$} if, for every $0 < f \in
    E$, there exists a time $t_0 \geq 0$ such that $e^{tA}f \gg_u 0$ for
    all $t \geq t_0$.
  \item[(d)] The semigroup is called \emph{uniformly eventually strongly
      positive withh respect to $u$} if it is individually eventually
    strongly positive with respect to $u$ and if the time $t_0$ form (c)
    can be chosen to be independent of $f$.
  \end{enumerate}
\end{definition_no_number}
It was demonstrated in \cite[Examples~5.7 and~5.8]{Daners2016} that
individual eventual strong positivity does not in general imply uniform
eventual positivity (neither for resolvents nor for semigroups). In
finite dimensions however, each of the above individual notions
coincides with its uniform counterpart and we shall thus only speak of
\emph{eventual positivity} and \emph{eventual strong positivity} if we
work on finite dimensional Banach lattices (where the quasi-interior
point $u$ is not mentioned explicitly in the latter notion since the
question whether a resolvent or a semigroup is eventually strongly
positive with respect to $u$ does not depend on $u$ in finite
dimensions).

In the present paper we mainly deal with eventual strong positivity with
respect to a given quasi-interior point $u$ (which is much easier to
characterise than mere eventual positivity, as observed in
\cite[Examples~7.1]{Daners2016a}). Mere eventual positivity will,
however, occur in several counterexamples in this article.

\section{Losing eventual positivity under positive perturbations}%
\label{section:counterexamples}%
If $(e^{tA})_{t \geq 0}$ is a positive $C_0$-semigroup on a complex
Banach lattice $E$ (meaning that $e^{tA} \geq 0$ for all $t \geq 0$) and
$B \in \calL(E)$ is a positive operator, then it follows easily from the
Dyson--Phillips series (see e.g.\ \cite[Theorem~III.1.10]{Engel2000})
that the perturbed semigroup $(e^{t(A+B)})_{t \geq 0}$ is positive,
too. If, on the other hand, $B \in \calL(E)$ is not necessarily
positive, but real and a multiplication operator, then we can also
conclude that $(e^{t(A+B)})_{t \geq 0}$ is positive. Indeed, we have $B +
c \geq 0$ for a sufficiently large number $c \geq 0$ and hence,
\begin{displaymath}
  e^{t(A+B)} = e^{-c} e^{t(A+B+c)} \geq 0
\end{displaymath}
for all $t \geq 0$. It is the purpose of the current section to
demonstrate that matters are much more complicated for eventually
positive semigroups. In the first subsection we show how eventual
positivity of the semigroup can get lost if we perturb $A$ by a
sufficiently large positive operator. In the second
subsection we demonstrate that \emph{individual} eventual positivity can
be destroyed by positive perturbations of arbitrarily small norm.

\subsection{Large perturbations}
\label{sec:largep}
It was recently demonstrated by Shakeri and Alizadeh 
\cite[Proposition~3.6]{Shakeri2017} that eventual strong positivity of a matrix 
can always destroyed be a positive perturbation, unless the original matrix
was positive itself. A similar phenomenon occurs for $C_0$-semigroups.
We first illustrate this by a concrete three dimensional example
(Example~\ref{ex:positive-perturbation}). Afterwards we prove a general 
theorem which shows that the situation is similar in infinite
dimensions (Theorem~\ref{thm:negative-result-for-large-perturbations}).

Let us now begin by studying a simple three dimensional matrix $A$
that generates an eventually strongly positive semigroup on $\bbC^3$.
We will show that the eventual positivity is destroyed if we perturb
$A$ by a certain positive multiplication operator (i.e.\ by a diagonal
matrix whose entries are all $\geq 0$). Our example is a manifestation of
the fact that certain sign patterns may or may not lead to eventual
positivity as extensively discussed in \cite{Berman2009,Erickson2015}
and references therein.

\begin{example}
  \label{ex:positive-perturbation}
  We consider the symmetric matrix
  \begin{equation}
    \label{eq:Amatrix}
    A=
    \begin{bmatrix}
      -2&-1&3\\
      -1&-2&3\\
      3& 3&-6
    \end{bmatrix}
  \end{equation}
  whose spectrum is $\sigma(A)=\{0,-1,-9\}$ and whose corresponding
  eigenvectors
  \begin{displaymath}
    u_1=\frac{1}{\sqrt{3}}
    \begin{bmatrix}
      1\\1\\1
    \end{bmatrix},\qquad u_2=\frac{1}{\sqrt{2}}
    \begin{bmatrix}
      1\\-1\\0
    \end{bmatrix}\quad\text{and}\quad
    u_3=\frac{1}{\sqrt{6}}
    \begin{bmatrix}
      1\\1\\-2
    \end{bmatrix},
  \end{displaymath}
  form an orthonormal basis in $\bbC^3$. Hence $0$ is the dominant
  eigenvalue of $A$; the corresponding eigenspace $\ker A$ is
  one-dimensional and contains an eigenvector whose entries are all
  strictly positive. It thus follows from
  \cite[Theorem~6.7]{Daners2016a} that the semigroup $(e^{tA})_{t\geq
    0}$ is eventually strongly positive. Yet, the semigroup is not
  positive because $A$ has negative entries outside the diagonal.  We
  now show that a self-adjoint rank-1 perturbation of the form $sB$ with
  $s>0$ and
  \begin{displaymath}
    B:=
    \begin{bmatrix}
      0&0&0\\
      0&1&0\\
      0&0&0
    \end{bmatrix}
  \end{displaymath}
  destroys the eventual positivity if $s>4$. Indeed, it is easily
  verified that
  \begin{displaymath}
    v=
    \begin{bmatrix}
      0\\3\\1
    \end{bmatrix}
  \end{displaymath}
  is an eigenvector of $(A+4B)$ corresponding to the eigenvalue
  $3$. Computing the other eigenvalues we obtain
  \begin{displaymath}
    \sigma(A+4B)=\Bigl\{3,-\frac{1}{2}\Bigl(9\pm\sqrt{65}\Bigr)\Bigr\},
  \end{displaymath}
  so $3$ is the dominant eigenvalue. For $s=4$ the eigenfunction is not
  strongly positive any more, and we will show that by choosing $s>4$
  the positivity is lost entierly.

  Since all eigenvalues are simple, it follows from standard
  perturbation theory that there exists a curve $\lambda(s)$ and vectors
  $u(s)\neq 0$ depending analytically on $s$ in an open interval $J$
  containing $s=4$, such that $\lambda(s)u(s)=(A+sB)u(s)$ for all $s\in
  J$ with initial conditions $\lambda(4)=3$ and $u(4)=v$; see
  \cite[Section~II.1.7]{Kato1976}. Differentiating the above equation
  with respect to $s$ yields
  \begin{equation}
    \label{eq:evp-diff}
    \lambda'(s)u(s)+\lambda(s)u'(s)=Bu(s)+(A+sB)u'(s).
  \end{equation}
  Taking the inner product of \eqref{eq:evp-diff} with $u(s)$ and using
  the symmetry of $A+sB$ we see that
  \begin{multline*}
    \lambda'(s)\|u(s)\|^2+\lambda(s)\langle u'(s),u(s)\rangle
    =\langle Bu(s),u(s)\rangle+\langle (A+sB)u'(s),u(s)\rangle\\
    =\langle Bu(s),u(s)\rangle+\langle u'(s),(A+sB)u(s)\rangle =\langle
    Bu(s),u(s)\rangle+\lambda(s)\langle u'(s),u(s)\rangle
  \end{multline*}
  and so
  \begin{displaymath}
    \lambda'(s)=\frac{\langle Bu(s),u(s)\rangle}{\|u(s)\|^2}
  \end{displaymath}
  for all $s\in J$. If we apply this to $s=4$ we obtain
  \begin{equation}
    \label{eq:evp-diff-4}
    \lambda'(4)
    =\frac{\langle Bv,v\rangle}{\|v\|^2}
    =\frac{9}{10}.
  \end{equation}
  To compute $w:=u'(4)$ we rearrange \eqref{eq:evp-diff} to get
  \begin{displaymath}
    \bigl(A+sB-\lambda(s)I\bigr)u'(s)=\bigl(\lambda'(s)I-B\bigr)u(s).
  \end{displaymath}
  Setting $s=4$ and making use of \eqref{eq:evp-diff-4}, we need to
  solve
  \begin{displaymath}
    (A+4B-3I)w=\Bigl(\frac{9}{10}I-B\Bigr)v.
  \end{displaymath}
  Substituting the matrices $A$ and $B$ we seek
  $w=(w_1,w_2,w_3)\in\bbR^3$ so that
  \begin{displaymath}
    \begin{bmatrix}
      -5&-1&3\\
      -1&-1&3\\
      3&3&-9
    \end{bmatrix}
    \begin{bmatrix}
      w_1\\w_2\\w_3
    \end{bmatrix}
    =\frac{1}{10}
    \begin{bmatrix}
      9&0&0\\
      0&-1&0\\
      0&0&9\\
    \end{bmatrix}
    \begin{bmatrix}
      0\\3\\1
    \end{bmatrix}
    =\frac{1}{10}
    \begin{bmatrix}
      0\\-3\\9
    \end{bmatrix}.
  \end{displaymath}
  Solving this equation we see that
  \begin{displaymath}
    w=u'(4)=\frac{1}{40}
    \begin{bmatrix}
      -3\\15\\0
    \end{bmatrix}
    +\tau
    \begin{bmatrix}
      0\\3\\1
    \end{bmatrix},
  \end{displaymath}
  for some $\tau\in\bbR$. Regardless of the value of $\tau$, the first
  component of $u(s)$ has a negative derivative at $s=4$, which means
  that first component changes sign from positive to negative at
  $s=4$. Hence the eigenvector $u(s)$ of the dominant eigenvalue
  $\lambda(s)$ is not positive (or negative) for $s$ in some interval
  $(4,4+\varepsilon)$, where $\varepsilon>0$. Hence, the semigroup
  $(e^{t(A+sB)})_{t\geq 0}$ is \emph{not} eventually positive for $s\in
  (4,4+\varepsilon)$. This follows for instance from
  \cite[Theorem~7.7(i)]{Daners2016}.
\end{example}

Next we look at the above example in a different way.

\begin{example}
  \label{ex:evpos-perturbation}
  Clearly the matrix
  \begin{displaymath}
    C_{a,s}=
    \begin{bmatrix}
      a&a&a\\
      a&s&a\\
      a&a&a
    \end{bmatrix}
  \end{displaymath}
  generates a strongly positive semigroup $(e^{tC_{a,s}})_{t\geq 0}$ on
  $\bbC^3$ for every $a,s>0$.  Let $A$ be given by
  \eqref{eq:Amatrix}. By Example~\ref{ex:positive-perturbation}
  $(e^{tA})_{t\geq 0}$ is eventually strongly positive but not
  positive. We have also seen in Example~\ref{ex:positive-perturbation}
  that for $a=0$ the semigroup $(e^{t(C_{a,s}+A)})_{t\geq 0}$ is not
  eventually positive for suitable choice of $s>4$. The reason is that
  the eigenvector corresponding to the dominant eigenvalue has strictly
  positive and strictly negative components. Having chosen such $s>4$,
  the continuous dependence of the eigenvalues and eigenvectors on the
  coefficients of a matrix shows that we can choose $a>0$ such that
  $(e^{t(C_{a,s}+A)})_{t\geq 0}$ is not eventually positive.

  Hence we have the generator $C_{a,s}$ of a strongly positive semigroup
  and a bounded operator $A$ generating an eventually strongly positive
  semigroup, but the semigroup generated by $C_{a,s}+A$ does not exhibit
  any positivity properties.
\end{example}

The above example demonstrates that strong positivity of a semigroup might 
be destroyed if the generator is perturbed by the generator of an eventually 
strongly positive semigroup; compare also \cite[Theorem~3.5]{Shakeri2017}.

We close this subsection with a general result asserting that, under
certain technical assumptions, eventual strong positivity of a semigroup
with respect to a quasi-interior point $u$ is \emph{always} unstable
under suitable large positive perturbation unless the semigroup is
positive. Recall from \cite[Theorem~7.6]{Daners2016} that, if
$(e^{tA})_{t \geq 0}$ is an eventually positive $C_0$-semigroup and the
spectrum $\spec(A)$ is non-empty, then the spectrum contains the
spectral bound $\spb(A)$. A finite dimensional analogue of the following
theorem, which deals with powers of matrices rather than with time
continuous semigroups, can be found in
\cite[Proposition~3.6]{Shakeri2017}.

\begin{theorem}
  \label{thm:negative-result-for-large-perturbations}
  Let $E$ be a complex Banach lattice and let $(e^{tA})_{t \geq 0}$ be a
  real $C_0$-semigroup on $E$ which is individually eventually strongly
  positive with respect to a quasi-interior point $u$ of $E_+$. Suppose
  that $\spb(A)$ is not equal to $-\infty$ and that it is a pole of
  $\Res(\phdot,A)$. Then the following assertions are equivalent:
  \begin{enumerate}[\upshape (i)]
    \item For every positive operator $B \in \calL(E)$ the perturbed
      semigroup $(e^{t(A+B)})_{t \geq 0}$ is individually eventually
      positive.
    \item For every positive rank-$1$ operator $B \in \calL(E)$ the perturbed
      semigroup $(e^{t(A+B)})_{t \geq 0}$ is individually eventually
      positive.
    \item The semigroup $(e^{tA})_{t \geq 0}$ is positive.
  \end{enumerate}
\end{theorem}
\begin{proof}
  We may assume that $\spb(A) = 0$. Obviously, (iii) implies (i) and (i)
  implies (ii). To show ``(ii) $\Rightarrow$ (iii)'', assume that 
  $(e^{t(A+B)})_{t \geq 0}$ is individually eventually positive for
  every positive rank-$1$ operator $B \in \calL(E)$.
  
  It suffices to prove that
  $\Res(\mu,A) \geq 0$ for all $\mu > 0$. To this end, fix an
  arbitrary real number $\mu > 0$ and an arbitrary functional $0
  <\varphi \in E'$. We show that $\Res(\mu,A)'\varphi \geq 0$.
	
  Since the spectral value $\spb(A)=0$ is a pole of $\Res(\phdot,A)$, 
  it is an eigenvalue
  of $A$ \cite[Theorem~2 in Section~VIII.8]{Yosida1995}, and it follows
  from \cite[Theorem~5.1]{DanersNOTE} and \cite[Corollary~3.3]{Daners2016a} 
  that $A$ admits an eigenvector $v \gg_u 0$ for the eigenvalue $0$. Since
  $u$ is a quasi-interior point of $E_+$, so is $v$ and hence we have
  $\langle \varphi, v \rangle > 0$. We can thus find a scalar $\alpha >
  0$ such that $\alpha \langle \varphi, v \rangle = \mu$.
	
  Define $B := \alpha \varphi \otimes v \in \calL(E)$. As $B$ is a
  positive rank-$1$ operator, the semigroup $(e^{t(A+B)})_{t \geq 0}$ is by
  assumption individually eventually positive. It follows from
  Proposition~\ref{prop:rank-1-pert-special-case}(a) that
  $\spb(A+B)=\alpha\langle\varphi,v\rangle = \mu$, that this number
  is a first order pole of the resolvent $\Res(\phdot,A+B)$ and that the
  corresponding spectral projection $Q$ is given by $Q = (\varphi \otimes
  v) \Res(\mu,A) = \bigl(\Res(\mu,A)'\varphi\bigr) \otimes v$.
  
  Since $(\lambda - \mu)\Res(\lambda,A+B) \to Q$ with respect to the operator
  norm as $\lambda \downarrow 
  \mu$ and since the semigroup generated by $A+B$ is individually
  eventually positive, it follows from \cite[Corollary~7.3]{Daners2016}
  that $Q \ge 0$. Thus, we conclude that $\bigl(\Res(\mu,A)'
  \varphi\bigr) \otimes v \ge 0$ and hence, $\Res(\mu,A)'\varphi \ge 
  0$, as claimed.
\end{proof}

\subsection{Small perturbations}
In this subsection we demonstrate that \emph{individual} eventual
positivity is very unstable with respect to small perturbations. The
following example shows that it can be destroyed by positive
perturbations of arbitrarily small norm. To do all necessary
computations in our example we need a few formulas for
rank-$1$-perturbations which can be found in the appendix of the paper.
On any given set $S$ we denote the constant function $S \to \bbR$ with
value $1$ by $\one$.
\begin{example}
  \label{ex:infinite-dimensional}
  On the Banach lattice $E = C([-1,1])$ there exist a bounded linear
  operator $A$ and a positive rank-$1$-projection $K$ with the following
  properties:
  \begin{enumerate}[\upshape (a)]
  \item The spectral bound $\spb(A)$ equals $0$, is a dominant spectral
    value of $A$ and a first order pole of the resolvent $\Res(\phdot,A)$.
    
    For every $\alpha > 0$ the spectral bound $\spb(A+\alpha K)$ equals
    $\alpha$, is a dominant spectral value of $A + \alpha K$ and a first
    order pole of the resolvent $\Res(\phdot,A+\alpha K)$.
  \item The resolvent $\Res(\phdot,A)$ is individually but not uniformly
    eventually strongly positive with respect to $\one$ at $0$.
    
    Moreover, the semigroup $(e^{tA})_{t \geq 0}$ is individually but
    not uniformly eventually strongly positive with respect to $\one$.
  \item The resolvent $\Res(\phdot, A+\alpha K)$ is not individually
    eventually positive at $\spb(A+\alpha K)$ for any $\alpha > 0$.
    
    Moreover, the semigroup $(e^{t(A+\alpha K)})_{t \geq 0}$ is not
    individually eventually positive for any $\alpha > 0$.
  \end{enumerate}
  To prove this, we choose $A$ to be the same operator which was
  constructed in \cite[Example~5.7]{Daners2016}. For the convenience of the
  reader we briefly recall this construction:
	
  Let $\varphi \in E'$ be the functional given by $\langle \varphi, f
  \rangle = \int_{-1}^1 f dx$ for every $f \in E$ and let $F = \ker
  \varphi$. Then we have $E = \langle \one \rangle \oplus F$, where
  $\one$ denotes the constant function with value $1$ and $\langle \one
  \rangle$ is its span. Let $S \in \calL(F)$ be the \emph{reflection
    operator} given by $(Rf)(\omega) = f(-\omega)$ for every $f \in F$
  and every $\omega \in [-1,1]$ and let $A \in \calL(E)$ be given by
  \begin{displaymath}
    A = 0_{\langle \one \rangle} \oplus (-2\id_F - S).
  \end{displaymath}
  We define $K := \one \otimes \delta_{-1}$, where $\delta_{-1}$ is the
  Dirac functional $\delta_{-1}\colon f \mapsto f(-1)$ on $E$.  Hence,
  we have $Kf = f(-1)\one$ for every $f \in E$. Obviously, $K$ is a
  positive rank-$1$-projection. Let us now show that the properties
  (a)--(c) are fulfilled.

  (a) Since $\sigma(S) = \{-1,1\}$, we conclude that $\sigma(A) =
  \{-3,-1,0\}$.  Hence, the spectral bound $\spb(A)$ equals $0$ and is a
  dominant spectral value of $A$; clearly, it is also a first order pole
  of the resolvent $\Res(\phdot,A)$. Note that $\one$ is an eigenvector of
  $A$ for the eigenvalue $0$.
  	
  Now, let $\alpha > 0$. We have $\alpha K = \alpha \delta_{-1} \otimes
  \one$ and it follows from
  Proposition~\ref{prop:rank-1-pert-special-case}(a) that any complex
  number $\lambda$ with $\repart \lambda > 0$ is a spectral value of
  $A+\alpha K$ if and only if $\lambda = \langle \alpha \delta_{-1},
  \one \rangle = \alpha$. Hence, the spectral bound of $A+\alpha K$
  equals $\alpha$ and is a dominant spectral value of $A+\alpha K$. The
  formula for $\Res(\phdot,A+\alpha K)$ in
  Proposition~\ref{prop:rank-1-pert-special-case}(a) immediately shows
  that the spectral value $\alpha$ is a first order pole of the
  resolvent.
  
  (b) This was shown in \cite[Example~5.7]{Daners2016}.
  
  (c) Fix $\alpha > 0$. We argue similarly as in
  \cite[Example~5.7]{Daners2016}: for every $\varepsilon \in (0,1)$ we
  can find a function $0 \leq f_\varepsilon \in E$ such that
  $f_\varepsilon(1) = \|f_\varepsilon\|_\infty = 1$, $\langle \varphi,
  f_\varepsilon \rangle = \varepsilon$ and $f_\varepsilon(-1) = 0$.  A
  short computation (or compare with \cite[formula~(5.3) in
  Example~5.7]{Daners2016}) shows that the resolvent of $A$ is given by
  \begin{displaymath}
    \Res(\lambda,A)
    = \frac{1}{\lambda} \id_{\langle \one \rangle} \oplus \frac{1}{(\lambda+2)^2-1}
    \bigl((\lambda+2)\id_F -S\bigr)
  \end{displaymath}
  for every $\lambda \in \rho(A)$. Using this an elementary calculation
  yields
  \begin{displaymath}
    \bigl(\Res(\lambda,A)f_\varepsilon\bigr)(-1)
    = \frac{\varepsilon}{2} \bigl( \frac{1}{\lambda}
    - \frac{1}{\lambda+3} \bigr) - \frac{1}{(\lambda+2)^2- 1}
  \end{displaymath}
  for every $\lambda \in \rho(A)$. Hence, for every $\lambda > 0$ we can
  find an $\varepsilon \in (0,1)$ such that
  $\bigl(\Res(\lambda,A)f_\varepsilon\bigr)(-1) < 0$. Now we can show
  that $\Res(\phdot,A+\alpha K)$ is not individually eventually positive
  at $\spb(A+\alpha K) = \alpha$: According to
  formula~\eqref{eq:resolvent-rank-one-pert-special-case} we have
  \begin{displaymath}
    \Res(\lambda,A+\alpha K)f_\varepsilon
    = \Res(\lambda,A)f_\varepsilon 
    + \frac{\alpha}{\lambda-\alpha}\bigl(\Res(\lambda,A)f_\varepsilon\bigr)(-1)\one
  \end{displaymath}
  and thus
  \begin{displaymath}
    \bigl( \Res(\lambda,A+\alpha K)f_\varepsilon \bigr)(-1) 
    = (1+\frac{\alpha}{\lambda-\alpha})\bigl(\Res(\lambda,A)f_\varepsilon \bigr)(-1)
  \end{displaymath}
  for all $\lambda \in \rho(A+\alpha K)$. Hence, if $\lambda > \alpha$
  is given, then we only have to choose $\varepsilon > 0$ such that
  $\bigl(\Res(\lambda,A)f_\varepsilon\bigr)(-1) < 0$ to obtain
  $\Res(\lambda,A+\alpha K)f_\varepsilon \not\geq 0$.
  
  It only remains to show that the semigroup $(e^{t(A+\alpha K)})_{t \ge
    0}$ is not individually eventually positive. To this end, we choose
  $\varepsilon > 0$ such that $\bigl(\Res(\alpha,A)f_\varepsilon\bigr)(-1)
  < 0$. It follows from
  formula~\eqref{eq:semigroup-rank-one-pert-special-case} that we have
  \begin{displaymath}
    e^{-t\alpha} e^{t(A+\alpha K)}f_\varepsilon
    = e^{-t\alpha}e^{tA}f_\varepsilon 
    + \alpha \Big[ \bigl(\Res(\alpha,A)f_\varepsilon\bigr)(-1) 
    - \bigl(e^{-t\alpha}e^{tA}\Res(\alpha,A)f_\varepsilon\bigr)(-1) \Big]\one 
  \end{displaymath}
  for every $t \geq 0$. Since the spectral bound of $A-\alpha$ equals
  $-\alpha$ and the operator $A-\alpha$ is bounded, we have
  $e^{-t\alpha}e^{tA} \to 0$ as $t \to \infty$ with respect to the
  operator norm. Hence, $e^{-t\alpha}e^{-t\alpha} e^{t(A+\alpha
    K)}f_\varepsilon$ converges to
  $\alpha \bigl(\Res(\alpha,A)f_\varepsilon\bigr)(-1) \one < 0$ with respect to
  the $\|\cdot\|_\infty$-norm as $t \to \infty$. In particular,
  $e^{t(A+\alpha K)}f_\varepsilon$ is not positive (in fact, it even
  fulfils $-e^{t(A+\alpha K)}f_\varepsilon \gg_{\one} 0$) 
  for all sufficiently large $t$. \qed
\end{example}

The above example indicates that if we want to prove any perturbation
results for eventually positive resolvents or semigroups, then we should
assume a version of \emph{uniform} eventual positivity. This is our
leitmotif for the rest of the paper.

\section{Perturbation theorems for resolvents}
\label{section:resolvents}

In this section we consider resolvents which are, at a spectral value
$\lambda_0$, uniformly eventually strongly positive with respect to a
quasi-interior point $u$. In the first subsection we show that this
property is stable with respect to sufficiently small perturbations
which are either positive or real multiplication operators. In the
second subsection we consider uniform eventual strong positivity at the
spectral bound and prove a \emph{quantitative} perturbation result for
this property.

A concrete class of operators for which eventual positivity of resolvents
has been studied for quite some time---though usually not under this name---is
constituted by fourth order differential operators (see~\cite{Sweers2016a} for an 
overview; compare also \cite[Proposition~6.5]{Daners2016a}). For such operators,
various perturbations results have been proved by quite concrete 
methods and estimates; see for instance~\cite{GrSw96}. Here, we rather focus on 
abstract functional analytical tools and prove results for abstract operators.

\subsection{A qualitative result}
The main result of this subsection is the following qualitative
perturbation result on eventually strongly positive resolvents at
arbitrary real eigenvalues which are poles of the resolvent.  Note that
we do not make any kind of compactness assumption in this theorem.
\begin{theorem}
  \label{thm:resolvent-small-perturbation}
  Let $E$ be a complex Banach lattice and let $A$ be a closed, densely
  defined real operator on $E$. Assume that $\lambda_0\in \sigma(A)\cap
  \bbR$ is a pole of $\Res(\phdot,A)$ and suppose that $\Res(\phdot,A)$ is
  uniformly eventually strongly positive with respect to $u$ at $\lambda_0$, where
  $u$ is a quasi-interior point of $E_+$.
  
  For all sufficiently small $r>0$ there exists $\varepsilon>0$ such
  that the following properties hold for every positive operator $B \in
  \calL(E)$ of norm $\|B\| < \varepsilon$:
  \begin{enumerate}[\upshape (a)]
  \item The operator $A+B$ has a unique spectral value $\lambda_B\in
    B(\lambda_0,r)$.
  \item The spectral value $\lambda_B$ is a real number, a pole of the
    resolvent $\Res(\phdot,A+B)$ and an algebraically simple eigenvalue
    of $A+B$.
  \item The resolvent $\Res(\phdot,A+B)$ is uniformly eventually strongly
    positive with respect to $u$ at $\lambda_B$.
  \end{enumerate}
\end{theorem}
One can prove a similar result for perturbations $B$ which are not
positive, but real multiplication operators; see
Corollary~\ref{cor:resolvent-small-perturbation} below.

In order to prove Theorem~\ref{thm:resolvent-small-perturbation} we need
two auxiliary results. The first one is a version of
\cite[Proposition~4.2]{Daners2016} on arbitrary Banach lattices. The
fact that such a result holds was already remarked in the discussion
after \cite[Definition~4.2]{Daners2016a}; however, the result was not
stated explicitly there.
\begin{proposition}
  \label{prop:uniform-eventual-positivity-by-neumann-series}
  Let $A\colon E \supseteq D(A) \to E$ be a real operator on a complex
  Banach lattice $A$ and let $\lambda_0$ be either $-\infty$ or a spectral 
  value of $A$ in $\bbR$. Consider a real number $\lambda_1 >
  \lambda_0$ such that $(\lambda_0, \lambda_1] \subseteq \resSet(A)$ and
  assume that $\Res(\lambda_1,A) \geq 0$. Then the following assertions
  hold.
  \begin{enumerate}[\upshape (a)]
  \item We have $\Res(\lambda,A) \geq 0$ for all $\lambda \in (\lambda_0,
    \lambda_1]$.
  \item If $u$ is a quasi-interior point of $E_+$ and if
    $\Res(\lambda_1,A)^n \gg_u 0$ for some $n \in \bbN$, then
    $\Res(\lambda,A) \gg_u 0$ for all $\lambda \in (\lambda_0,
    \lambda_1)$.
  \end{enumerate}
\end{proposition}
\begin{proof}
  The proof is exactly the same as the proof of
  \cite[Proposition~4.2]{Daners2016}.
\end{proof}
The second ingredient for the proof of
Theorem~\ref{thm:resolvent-small-perturbation} is 
Lemma~\ref{lem:eigenvalue-perturbation} below that
guarantees that a pole of the resolvent which is, in addition, an
algebraically simple real eigenvalue preserves these properties through
a small perturbation by a real operator. This lemma is a typical 
result from standard perturbation theory (compare for instance
\cite[Section~IV.3]{Kato1976}). Though, in order to have it 
available in exactly the version we need, we include a proof.
In the preliminaries we
introduced the concept of a real operator only on complex Banach
lattices and to avoid the necessity of even more terminology, we shall
state the lemma only on those spaces; compare however
Remark~\ref{rem:complexifications} below.

\begin{lemma}
  \label{lem:eigenvalue-perturbation}
  Let $E$ be a complex Banach lattice and let $A$ be a closed operator
  on $E$. Assume that $\lambda_0\in\spec(A)$ is a pole of the resolvent
  $\Res(\phdot,A)$ and an algebraically simple eigenvalue of $A$ with
  spectral projection $P_0$.  Let $r>0$ be such that
  $\overline{B(\lambda_0,r)}\cap \sigma(A)=\{\lambda_0\}$ and set
  $\varepsilon =\min_{|\lambda-\lambda_0|=r}\|\Res(\lambda,A)\|^{-1}$.  For
  every $B \in \calL(E)$ with $\|B\| < \varepsilon$ the following
  assertions are fulfilled:
  \begin{enumerate}[\upshape (a)]
  \item $A+B$ has a unique spectral value $\lambda_B\in B(\lambda_0,r)$
    and $\lambda_B$ is a pole of the resolvent $\Res(\phdot,A+B)$
    and an algebraically simple eigenvalue of $A+B$.
  \item Denote by $P_B$ the spectral projections associated with
    $\lambda_B$. Then $\lambda_B\to\lambda_0$ and $P_B\to P_0$ with
    respect to the operator norm as $\|B\| \to 0$.
  \item If $\lambda_0 \in \bbR$ and the operators $A$ and $B$ are real,
    then $\lambda_B\in\bbR$.
  \end{enumerate}
\end{lemma}
\begin{proof}
  Let $C_r$ be the
  positively oriented circle of radius $r>0$ centred at $\lambda_0$ as
  given in the statement of the lemma and let $B \in \calL(E)$ with $\|B\| 
  < \varepsilon$. For all $\lambda \in C_r$ we have
  $\|\Res(\lambda,A)B\| \leq \|\Res(\lambda,A)\| \|B\| \leq \|B\|/\varepsilon
  < 1$.  Since $\lambda I-(A+B)=\bigl(I-B\Res(\lambda,A)\bigr)(\lambda I-A)$,
  a Neumann series expansion yields that $\lambda I - (A+B)$ is
  invertible and that
  \begin{equation}
    \label{eq:perturbed-resolvent}
    \Res(\lambda,A+B)
    =\Res(\lambda,A)\bigl[I-B\Res(\lambda,A)\bigr]^{-1}
    =\Res(\lambda,A)\sum_{k=0}^\infty\bigl[B\Res(\lambda,A)\bigr]^k
  \end{equation}
  for all $\lambda \in C_r$. In particular we can define the projection
  \begin{equation*}
    P_B:=\frac{1}{2\pi i}\int_{C_r}\Res(\lambda,A+B)\,d\lambda.
  \end{equation*}
  We first show that $P_B$ depends continuously on $B$. Indeed, let 
  $\alpha:= \min_{|\lambda-\lambda_0|=r} \|\Res(\lambda,A+B)\|^{-1}$, 
  i.e.\ we have $\alpha \cdot \|\Res(\lambda,A+B)\| \le 1$ for all $\lambda \in C_r$.
  Let $\delta \in (0,1)$. Another Neumann series argument shows that
  whenever an operator $\tilde B \in \calL(E)$, say of norm $\|\tilde B\| < 
  \varepsilon$, is closer to $B$ than $\alpha\delta$, then
  \begin{displaymath}
  		\|\Res(\lambda, A+\tilde B) - \Res(\lambda,A+B)\| \le 
  		\frac{\delta}{\alpha(1-\delta)}
  \end{displaymath}
  for all $\lambda \in C_r$, and thus $\|P_{\tilde B} - P_{B}\| \le 
  \delta \frac{r}{\alpha(1-\delta)}$. This proves that $P_{\tilde B}
  \to P_B$ for $\tilde B \to B$.
  
  Now it follows from \cite[Lemma~I.4.10]{Kato1976} (the proof there does not rely on
  $E$ being finite dimensional) and our assumption that $\dim(\im
  P_B)=\dim(\im P_0)=1$ whenever $\|B\|<\varepsilon$. In particular,
  $A+B$ has only one spectral value $\lambda_B$ in the disk
  $B(\lambda_0,r)$; since the corresponding spectral projection $P_B$
  has rank one, it follows that $\lambda_B$ is a pole of the resolvent
  $\Res(\phdot,A+B)$ \cite[Section~III.6.5]{Kato1976} and an algebraically
  simple eigenvalue. We thus proved (a) and the second part of
  (b). Because $r>0$ can be chosen arbitrarily small, we conclude that 
  $\lambda_B\to 0$ as $\|B\|\to 0$, which proves the first part of (b).

  To prove (c), suppose that $A,B$ are real and that $\lambda_0 \in
  \bbR$.  If $\lambda_B\not\in\bbR$, then $\overline{\lambda}_B$ is a
  second spectral value of $A+B$ in the disk $B(\lambda_0,r)$, which
  contradicts (a). Thus, $\lambda_B \in \bbR$.
\end{proof}

\begin{remark}
  \label{rem:complexifications}
  The proof of Lemma~\ref{lem:eigenvalue-perturbation} actually shows a
  bit more. Assertions~(a) and~(b) of the lemma remain true if $E$ is
  only assumed to be a complex Banach space. Assertion~(c) does not make
  sense if $E$ is only a complex Banach space since the notion of a
  \emph{real operator} is not defined on such spaces. If, however, $E$ is a
  so-called \emph{complexification} of a real Banach space $E_\bbR$,
  then the notion of a real operator makes sense; in this situation,
  assertion~(c) of Lemma~\ref{lem:eigenvalue-perturbation} remains true.
  
  For a detailed treatement of complexifications we refer the reader for
  example to \cite{Munoz1999}. Here we only point out that every complex
  Banach lattice is a certain complexification of a real Banach lattice
  and thus of a real Banach space (see
  \cite[Section~II.11]{Schaefer1974} or
  \cite[Section~2.2]{Meyer-Nieberg1991}).
\end{remark}

We are now ready to prove
Theorem~\ref{thm:resolvent-small-perturbation}.

\begin{proof}[Proof of Theorem~\ref{thm:resolvent-small-perturbation}]
  It follows from \cite[Theorem~4.2 and Proposition~4.1]{DanersNOTE}
  that $\lambda_0$ is an algebraically simple eigenvalue of $A$. By
  assumption we can choose $r>0$ such that $\overline{B(\lambda_0,r)}
  \cap \sigma(A) = \{\lambda_0\}$ and that $\Res(\lambda_0+r,A)\gg_u 0$.
  Choose $\varepsilon > 0$ as in Lemma~\ref{lem:eigenvalue-perturbation}
  and let $B \in \calL(E)$ be positive with norm $\|B\| < \varepsilon$.
  Then by that lemma there
  exists a unique $\lambda_B\in B(\lambda_0,r)\cap\sigma(A+B)$. Moreover
  $\lambda_B$ is a pole of the resolvent $\Res(\phdot,A+B)$ and an
  algebraically simple eigenvalue of $A+B$ and we have
  $\lambda_B\in(\lambda_0-r,\lambda_0+r)$. This proves~(a) and~(b).
  
  Since $B\Res(\lambda_0+r,A)\geq 0$, identity
  \eqref{eq:perturbed-resolvent} with $\lambda$ replaced with
  $r+\lambda_0$ implies that $\Res(\lambda_0+r,A+B)\geq
  \Res(\lambda_0+r,A)\gg_u 0$.
  Proposition~\ref{prop:uniform-eventual-positivity-by-neumann-series}(b)
  now shows that $\Res(\lambda,A+B) \gg_u 0$ for all $\lambda \in
  (\lambda_B,\lambda_0 + r]$.
\end{proof}
Let us now consider the case where the perturbation $B$ is not positive,
but a real multiplication operator.
\begin{corollary}
  \label{cor:resolvent-small-perturbation}
  Let $E$ be a complex Banach lattice and let $A$ be a closed, densely
  defined real operator on $E$. Assume that $\lambda_0\in \sigma(A)\cap
  \bbR$ is a pole of $\Res(\phdot,A)$ and suppose that $\Res(\phdot,A)$ is
  uniformly eventually strongly positive with respect to $u$ at $\lambda_0$,
  where $u$ is a quasi-interior point of $E_+$.
  
  Then, for all sufficiently small $r>0$ there exists $\varepsilon>0$
  such that the assertions~{\upshape (a)--(c)} from
  Theorem~\ref{thm:resolvent-small-perturbation} hold for every real
  multiplication operator $B \in \calL(E)$ of norm
  $\|B\| < \varepsilon$.
\end{corollary}
The proof of this corollary relies on the following observation
concerning multiplication operators.
\begin{lemma}
  \label{lem:norm-of-multiplication-operator}
  Let $E$ be a complex Banach lattice and let $T \in \calL(E)$ be a real
  multiplication operator. Then we have
  \begin{displaymath}
    \|T\| \geq \min\{c \geq 0\colon -c\id \leq T \leq c\id\}.
  \end{displaymath}
\end{lemma}
\begin{proof}
  By $T_\bbR$ we denote the restriction $T|_{E_\bbR}$ of $T$ to the real
  part of $E_\bbR$ of $E$. We have
  \begin{displaymath}
    \min\{c \geq 0\colon -c\id_E \leq T \leq c\id_E\}
    = \min\{c \geq 0\colon-c\id_{E_\bbR} \leq T_\bbR \leq c\id_{E_\bbR}\}
    = \|T_\bbR\|,
  \end{displaymath}
  where the latter equality can be found in
  \cite[Theorem~3.1.11]{Meyer-Nieberg1991}.  This proves the assertion
  since we clearly have $\|T_\bbR\| \leq \|T\|$.
\end{proof}
\begin{remark}
  We suspect that there is equality in
  Lemma~\ref{lem:norm-of-multiplication-operator} as is true on real
  Banach lattices \cite[Theorem~3.1.11]{Meyer-Nieberg1991}. This is,
  however, not important for our purposes.
\end{remark}
We are now ready to prove
Corollary~\ref{cor:resolvent-small-perturbation}.
\begin{proof}[Proof of Corollary~\ref{cor:resolvent-small-perturbation}]
  According to Theorem~\ref{thm:resolvent-small-perturbation} we can,
  for each sufficiently small $r > 0$, find $\tilde \varepsilon > 0$
  such that for each operator $0 \leq \tilde B \in \calL(E)$ of norm
  $\|\tilde B\| < \tilde \varepsilon$ the following holds: there exists
  exactly one spectral value of $A+\tilde B$ in the disk
  $B(\lambda_0,r/3)$ and no other spectral value in the disk
  $B(\lambda_0, r)$ and the assertions~(b) and~(c) of the theorem are
  fulfilled for this spectral value and for the operator $A + \tilde B$.
	
  Now, define $\varepsilon := \min\{r/3, \tilde \varepsilon/2\}$ and let $B \in
  \calL(E)$ be a real multiplication operator of norm $\|B\| <
  \varepsilon$.  According to
  Lemma~\ref{lem:norm-of-multiplication-operator} we have $\tilde B := B
  + \|B\|\id \geq 0$; moreover, the positive operator $\tilde B$ has norm
  $\|\tilde B\| < \tilde \varepsilon$.
	
  Hence there exists exactly one spectral value of $A + \tilde B$ in the
  disk $B(\lambda_0,r/3)$ and no other spectral value in the disk
  $B(\lambda_0, r)$; furthermore, assertions~(b) and~(c) of
  Theorem~\ref{thm:resolvent-small-perturbation} are fulfilled for this
  spectral value and for the operator $A + \tilde B$.  This implies that
  the operator $A + \tilde B - \|B\|\id = A + B$ has exactly one
  spectral value in the disk $B(\lambda_0, 2r/3)$ and that
  assertions~(b) and~(c) of
  Theorem~\ref{thm:resolvent-small-perturbation} are fulfilled for this
  spectral value and for the operator $A+B$.
\end{proof}

For matrices we can prove a stronger result than
Theorem~\ref{thm:resolvent-small-perturbation}. Relying on the fact that
the set of strongly positive matrices is open in the space of all real
matrices we obtain stability of eventual strong positivity even with
respect to negative perturbations.

\begin{proposition}
  \label{prop:matrix-resolvents}
  The set of matrices in $\bbR^{d\times d}$ having an eigenvalue at
  which its resolvent is eventually strongly positive is open in
  $\bbR^{d\times d}$.
\end{proposition}
\begin{proof}
  Let $A\in\bbR^{d\times d}$ be a matrix having an eigenvalue
  $\lambda_0$ at which $\Res(\phdot,A)$ is eventually strongly positive. By
  \cite[Theorem 4.4]{Daners2016} the corresponding spectral projection
  $P_0$ fulfils $P_0 \gg 0$, by which we mean that every entry of $P_0$
  is strictly positive.  Moreover, according to \cite[Proposition
  3.1]{Daners2016} $\lambda_0$ is the only eigenvalue of $A$ having a
  positive eigenvector, and $\lambda_0$ is algebraically
  simple. Lemma~\ref{lem:eigenvalue-perturbation} implies the existence
  of $\varepsilon_0>0$ such that $A+B$ has an algebraically simple
  eigenvalue $\lambda_B \in \bbR$ near $\lambda_0$ if
  $\|B\|<\varepsilon_0$. Moreover, the corresponding spectral projection
  $P_B$ converges to $P_0$ as $\|B\|\to 0$. Since $P_0\gg 0$ there
  exists $\varepsilon\in(0,\varepsilon_0]$ such that $P_B\gg 0$ whenever
  $\|B\|<\varepsilon$. Now, \cite[Theorem 4.4]{Daners2016} implies that
  $\Res(\phdot,A+B)$ is eventually strongly positive at $\lambda_B$ whenever
  $\|B\|<\varepsilon$. Hence all matrices in the
  $\varepsilon$-neigbourhood of $A$ have an eigenvalue at which
  their resolvent is eventually strongly positive.
\end{proof}

\subsection{A quantitative result}

In this subsection we consider uniform eventual strong positivity of
resolvents at the spectral bound of an operator $A$ and prove a
quantitative perturbation result, meaning that we give an estimate of
how large a positive perturbation may be in norm in order not to destroy
the eventual strong positivity. Eventual positivity at $\spb(A)$ is of
particular importance since it is related to eventual positivity of
$(e^{tA})_{t \geq 0}$ (in case that $A$ is a generator); compare the
perturbation result in
Theorem~\ref{thm:quantitative-sg-perturbation-on-hilbert-space} which we
obtain as a consequence of the perturbation result in the present
subsection.

To formulate the next theorem we need the following notation: For every
operator $A\colon E \supseteq D(A) \to E$ on a complex Banach space $E$
we define the \emph{real spectral bound} $\spb_\bbR(A)$ of $A$ to be the
supremum of all real spectral values of $A$, i.e.\ $\spb_\bbR(A) :=
\sup(\sigma(A) \cap \bbR)$; we clearly have $-\infty \leq \spb_\bbR(A)
\leq \spb(A) \leq \infty$.

\begin{theorem}
  \label{theorem:a-compact-perturbation}
  Let $E\not=\{0\}$ be a complex Banach lattice, let $u \in E$ be a
  quasi-interior point of $E_+$ and let $A$ be a densely defined and
  real linear operator on $E$ such that $\spb(A)$ is a spectral value of
  $A$ and a pole of $\Res(\phdot,A)$.  Suppose there exists $\lambda_1 >
  \spb(A)$ such that $\Res(\lambda,A) \gg_u 0$ for all $\lambda \in
  (\spb(A), \lambda_1)$ and assume that $M := \sup_{\repart \lambda \geq
    \lambda_1} \|\Res(\lambda,A)\| < \infty$.
  
  Then, for every operator $0 \leq K \in \calL(E)$ with norm $\|K\| <
  \frac{1}{M}$ the real spectral bound $\spb_\bbR(A+K)$ fulfils the
  following properties:
  \begin{enumerate}[\upshape (i)]
  \item $\spb_\bbR(A+K) \leq \spb(A+K) < \lambda_1$.
  \item $\Res(\lambda,A+K) \gg_u 0$ for all $\lambda \in (\spb_\bbR(A+K),
    \lambda_1)$.
  \item If $K$ is $A$-compact, then $\spb_\bbR(A+K) \geq \spb(A)$ and
    $\spb_\bbR(A)$ is a pole of $\Res(\phdot,A+K)$.
  \item If $K$ is $A$-compact and non-zero, then $\spb_\bbR(A+K) >
    \spb(A)$.
  \end{enumerate}
\end{theorem}

Note that the assumption $M < \infty$ in the above theorem is
automatically fulfilled if $A$ generates a $C_0$-semigroup whose
growth bound coincides with $\spb(A)$ (recall
again that the latter is for example fulfilled if $(e^{tA})_{t \geq 0}$
is eventually norm continuous \cite[Corollary~IV.3.11]{Engel2000});
indeed, we have $\sup_{\repart \lambda \geq \tilde \lambda_1}
\|\Res(\lambda,A)\| < \infty$ for every $\tilde \lambda_1 > \spb(A)$ in
this case.

\begin{proof}[Proof of Theorem~\ref{theorem:a-compact-perturbation}]
  Let $0 \leq K \in \calL(H)$ and note that
  \begin{displaymath}
    \lambda I - (A + K) = \bigl[I - K \Res(\lambda,A)\bigr] (\lambda I-A).
  \end{displaymath}
  for each $\lambda \in \rho(A)$. Hence, if $\lambda \in \rho(A)$ and
  $\spr\bigl(K \Res(\lambda,A)\bigr) < 1$, then $\lambda \in \rho(A+K)$ and
  \begin{equation}
    \label{eq:resolvent-perturbation}
    \Res(\lambda,A + K) = \Res(\lambda,A)\bigl[I - K \Res(\lambda,A)\bigr]^{-1}.
  \end{equation}

  (i) Obviously, $\spb_\bbR(A+K) \leq \spb(A+K)$. If $\lambda \in \bbC$
  with $\repart \lambda \geq \lambda_1$ then we have $\spr(K
  \Res(\lambda,A)) \leq \|K \Res(\lambda,A)\| < \frac{1}{M} \|\Res(\lambda,A)\|
  \leq 1$ and thus $\lambda \in \rho(A+K)$. This proves that $\spb(A+K) <
  \lambda_1$.
  
  (ii) As noted in the proof of (i) we have $\|K \Res(\lambda_1,A)\| <
  1$.  Since the mapping $\lambda \mapsto \|K \Res(\lambda,A)\|$ is
  continuous we also have $\spr\bigl(K \Res(\lambda,A)\bigr) \leq \|K
  \Res(\lambda,A)\| < 1$ for all $\lambda \in
  (\lambda_1-\varepsilon,\lambda_1)$ if $\varepsilon > 0$ is chosen
  small enough. Each such $\lambda$ is contained in $\rho(A+K)$ and
  \eqref{eq:resolvent-perturbation} holds. Since $K \Res(\lambda,A)$ is
  positive and has spectral radius $<1$, the inverse $\bigl[I - K
  \Res(\lambda,A)\bigr]^{-1}$ is also positive. We thus have $\bigl[I -
  K \Res(\lambda,A)\bigr]^{-1}f > 0$ for each $f > 0$. As
  $\Res(\lambda,A) \gg_u 0$, formula~\eqref{eq:resolvent-perturbation}
  now yields $\Res(\lambda,A+K) \gg_u 0$. According to 
  Proposition~\ref{prop:uniform-eventual-positivity-by-neumann-series}(b)
  this implies that $\Res(\lambda,A+K) \gg_u 0$ holds in fact for all $\lambda \in
  (\spb_\bbR(A),\lambda_1)$.
  
  (iv) Assume now in addition that $K$ is $A$-compact and non-zero. To
  prove (iv) it suffices to show that $A+K$ has a spectral value
  $\lambda \in (\spb(A), \lambda_1)$.  To this end, let $P \in \calL(E)$
  be the spectral projection of $A$ associated with $\spb(A)$. By
  \cite[Theorem~4.1]{DanersNOTE} and \cite[Corollary~3.3]{Daners2016a}, 
  $P$ is a rank-$1$
  operator which fulfils $P \gg_u 0$ and we have $P = \lim_{\lambda
    \downarrow \spb(A)} (\lambda - \spb(A))\Res(\lambda,A)$ with respect to
  the operator norm. Let us now define a mapping $\gamma\colon [\spb(A),
  \lambda_1) \to \calL(E)$ which is given by
  \begin{displaymath}
    \gamma(\lambda) = 
    \begin{cases}
      \bigl(\lambda - \spb(A)\bigr) K \Res(\lambda,A)
      & \text{if $\lambda > \spb(A)$,} \\
      K P
      & \text{if $\lambda = \spb(A)$.}
    \end{cases}
  \end{displaymath}
  Note that $\gamma$ is continuous with respect to the operator norm and
  that $\gamma(\lambda)$ is a compact, positive operator for every
  $\lambda \in [\spb(A), \lambda_1)$. Moreover, we recall that the
  restriction of the mapping $\calL(E) \to [0,\infty)$, $T \mapsto
  \spr(T)$ to the set of compact operators is continuous with respect to
  the operator norm; this follows e.g.\ from \cite[Remark~IV.3.3 and the
  discussion in Section~IV.3.5]{Kato1976} or from
  \cite[Theorem~2.1(a)]{Degla2008}. Hence, $\spr(\gamma(\phdot))\colon
  [\spb(A), \lambda_1) \to [0,\infty)$ is continuous.
  
  Let us show that $\spr\bigl(\gamma(\spb(A)\bigr) = \spr(KP)>0$. Since
  $P$ has rank $1$ and since $P \gg_u 0$, we can find a strictly
  positive functional $\varphi \in E'$ and a vector $0 \ll_u v \in E$
  such that $P = \varphi \otimes v$ and hence, $KP = \varphi \otimes
  Kv$.  Since $v$ is a quasi-interior point of $E_+$ and $K$ is
  non-zero, it follows that $Kv \neq 0$.  Using that $\varphi$ is
  strictly positive, we deduce that $\langle \varphi, Kv\rangle > 0$ and
  hence $\sigma(KP) = \sigma(\varphi \otimes Kv) \ni \langle \varphi, Kv
  \rangle > 0$.  Thus, $\spr(KP) > 0$.
  
  We conclude that for all sufficiently small $\lambda > \spb(A)$ we have that
  \begin{displaymath}
    \spr\bigl((\lambda-\spb(A))K \Res(\lambda,A)\bigr)
    \geq \frac{\spr(KP)}{2} > 0.
  \end{displaymath}
  We can thus find $\lambda \in (\spb(A), \lambda_1)$ such that
  $\spr\bigl(K\Res(\lambda,A)\bigr) > 1$. On the other hand we have
  $\|K\Res(\lambda_1,A)\| < 1$. Hence we have
  $\spr\bigl(K\Res(\lambda,A)\bigr) \leq \|K\Res(\lambda,A)\| < 1$ for
  all $\lambda \in (\spb(A), \lambda_1)$ which are sufficiently close to
  $\lambda_1$. Using again that the spectral radius is continuous on the
  compact operators with respect to the norm topology
  \cite[Theorem~2.1(a)]{Degla2008} we conclude from the intermediate
  value theorem that $\spr(K\Res(\lambda,A)) = 1$ for some $\lambda \in
  (\spb(A),\lambda_1)$. For this $\lambda$ the operator $\lambda I-A$ is
  invertible, but the operator $I - K \Res(\lambda,A)$ is not since the
  spectral radius of $K \Res(\lambda,A)$ is 
  contained in its spectrum (this is a general fact for positive operators,
  see \cite[Proposition~V.4.1]{Schaefer1974}). Hence, it
  follows from \eqref{eq:perturbed-resolvent} that $\lambda \in
  \sigma(A+K)$.
  
  (iii) Assume that $K$ is $A$-compact. If $K=0$, then assertion (iii)
  is obvious. If $K$ is non-zero, then it follows from (iv) that
  $\spb_\bbR(A+K) > \spb(A)$. We use
  formula~\eqref{eq:resolvent-perturbation} to prove that
  $\spb_\bbR(A+K)$ is a pole of $\Res(\phdot,A+K)$. Let $\Omega := \{z
  \in \bbC\colon \repart z > \spb(A)\}$. On this set, the mappings
  $\lambda \mapsto \Res(\lambda,A)$ and $\lambda \mapsto K
  \Res(\lambda,A)$ are analytic and the latter one takes only compact
  operators as its values. Since $I-K\Res(\lambda,A)$ is invertible for
  at least one $\lambda \in \Omega$, it follows from the so-called
  \emph{Analytic Fredholm Theorem} (see
  e.g.~\cite[Theorem~1]{Steinberg1968/1969}) that $\bigl[I -
  K\Res(\lambda,A)\bigr]^{-1}$ is meromorphic on $\Omega$. Hence,
  $\Res(\phdot,A+K)$ is either analytic at $\spb_\bbR(A+K)$ or it has a
  pole there; yet, since $\spb_\bbR(A+K)$ is, of course, a spectral value of
  $A+K$, the latter alternative must be true.
\end{proof}

\section{Perturbation theorems for semigroups}
\label{section:semigroups}

In this final section we consider perturbations of semigroup generators.
We do however \emph{not} prove theorems of the type ``If $(e^{tA})_{t
  \geq 0}$ is eventually strongly positive, then so is $(e^{t(A+B)})_{t
  \geq 0}$ for appropriate $B$''. Those results would, of course, be
desirable, but it seems to be a difficult task to prove them. Instead we
assume that the resolvent of the semigroup generator $A$ is uniformly
eventually strongly positive at the spectral bound $\spb(A)$. Using the
results of Section~\ref{section:resolvents} we then show that the
resolvent of the perturbed operator $A+B$ is also uniformly eventually
strongly positive at $\spb(A+B)$ and, by means of the characterisation
results in \cite[Sections~4 and~5]{Daners2016a}, this yields
a least \emph{individual} eventual strong positivity of the semigroup 
generated by $A+B$. In case that the underlying space is an $L^2$-space,
one even obtains uniform eventual strong positivity of this semigroup, 
see Theorem~\ref{thm:quantitative-sg-perturbation-on-hilbert-space} below 
and \cite[Theorem~10.2.1]{GlueckDISS}.

\subsection{A qualitative result}
\label{sec:qualitative}
We start again with a subsection containing qualitative perturbation
results.  To prove our main theorems we need the following auxiliary
results. As we did with
Lemma~\ref{lem:eigenvalue-perturbation}, we only
formulate the following result on a complex Banach lattice, although the
proof shows that it is actually true on arbitrary complexifications of
real Banach spaces.

\begin{lemma}
  \label{lem:spectral-bound-perturbation}
  Let $E$ be a complex Banach lattice and let
  $(e^{tA})_{t \geq 0}$ be a real eventually norm continuous
  $C_0$-semigroup on $E$. Suppose furthermore that $\spb(A)$ is a
  dominant spectral value of $A$, a pole of the resolvent and an
  algebraically simple eigenvalue; denote the spectral projection
  associated with $\spb(A)$ by $P_0$. Then there exists an $\varepsilon >
  0$ such that the following properties are fulfilled for every real
  operator $B \in \calL(E)$ with $\|B\| < \varepsilon$:
  \begin{enumerate}[\upshape (a)]
  \item The spectral bound $\spb(A+B)$ of $A+B$ is a dominant spectral
    value of $A+B$, a pole of the resolvent and an algebraically simple
    eigenvalue.
  \item We have $\spb(A+B) \to \spb(A)$ and $P_B \to P_0$ with respect
    to the operator norm as $\|B\| \to 0$; here, $P_B$ denotes the
    spectral projection of $A+B$ associated with the isolated spectral
    value $\spb(A+B)$.
  \end{enumerate}
\end{lemma}
\begin{proof}
  Since
  $\spb(A)$ is a dominant spectral value and $(e^{tA})_{t \geq 0}$ is
  eventually norm continuous, we can find a number $r > 0$ such that
  $\repart \lambda \leq \spb(A) - 2r$ for all $\lambda \in \sigma(A)
  \setminus \{\spb(A)\}$. The spectral bound of the restriction of $A$
  to the kernel of $P_0$ fulfils $\spb(A|_{\ker P_0}) \leq \spb(A) - 2r$
  and since $(e^{tA}|_{\ker P_0})_{t \geq 0}$ is eventually norm
  continuous, it follows that the growth bound of this restricted
  semigroup is also no larger than $\spb(A)-2r$
  \cite[Corollary~IV.3.11]{Engel2000}. In particular, we obtain from the
  Laplace transform representation of the resolvent that
  \begin{displaymath}
    \sup_{\repart \lambda \geq \spb(A) - r} \|\Res(\lambda,A|_{\ker P_0})\| < \infty.
  \end{displaymath}
  On the other hand,
  \begin{displaymath}
    \sup_{|\lambda - \spb(A)| \geq r} \|\Res(\lambda,A|_{\im P_0})\| < \infty.
  \end{displaymath}
  Hence, $\|\Res(\phdot,A)\|$ is bounded by a constant $C \in (0,\infty)$
  on the set
  \begin{displaymath}
    \overline{\Omega}
    :=\bigl\{\lambda \in \bbC\colon
    \text{$\repart\lambda\geq\spb(A)-r$ and $|\lambda-\spb(A)|\geq r$}\bigr\}.
  \end{displaymath}
  Define $\varepsilon = \frac{1}{C}$ and let $B \in \calL(E)$ with
  $\|B\| < \varepsilon$. According to
  Lemma~\ref{lem:eigenvalue-perturbation}, $A+B$ has a uniquely
  determined spectral value $\lambda_B \in B(r,\spb(A))$, and this
  spectral value $\lambda_B$ is real, a pole of the resolvent $A+B$ and
  an algebraically simple eigenvalue of $A+B$. Moreover, $\lambda_B \to
  \spb(A)$ and $P_B \to P_0$ wit respect to the operator norm as $\|B\|
  \to 0$.
  
  It only remains to show that $A+B$ has no spectral value within the
  set $\overline{\Omega}$ since this implies that $\spb(A+B) =
  \lambda_B$ has the claimed properties. So, let $\lambda \in
  \overline{\Omega}$. Then we have
  \begin{displaymath}
    \lambda - (A+B) = \bigl[I - B\Res(\lambda,A)\bigr](\lambda I-A).
  \end{displaymath}
  Since $\|B\Res(\lambda,A)\| < \varepsilon C = 1$ this operator is
  invertible and hence, $\lambda\in\rho(A+B)$.
\end{proof}

Now we formulate and prove the first main result of this subsection.

\begin{theorem}
  \label{thm:qualitative-perturbation-sg}
  Let $E$ be a complex Banach lattice and let $(e^{tA})_{t \geq 0}$ be a
  real $C_0$-semigroup on $E$. Suppose that $\spb(A)$ is a dominant
  spectral value of $A$ and a pole of the resolvent. Suppose that
  $\Res(\phdot,A)$ is uniformly eventually strongly positive at
  $\spb(A)$ with respect to a quasi-interior point $u$ of $E_+$. Assume
  moreover that at least one of the following assumptions is fulfilled:
  \begin{enumerate}[\upshape (i)]
  \item $(e^{tA})_{t \geq 0}$ is analytic and $D(A) \subseteq E_u$.
  \item $(e^{tA})_{t \geq 0}$ is immediately norm-continuous and $E_u =
    E$.
  \end{enumerate}
  Then there exists an $\varepsilon > 0$ such that for every operator $0
  \leq B \in \calL(E)$ with $\|B\| < \varepsilon$ the semigroup generated
  by $A+B$ is individually eventually strongly positive with respect to
  $u$.
\end{theorem}
\begin{proof}
  According to Theorem~\ref{thm:resolvent-small-perturbation} and
  Lemma~\ref{lem:spectral-bound-perturbation} we can find an
  $\varepsilon > 0$ with the following property: for all $0 \leq B \in
  \calL(E)$ with $\|B\| < \varepsilon$ the spectral bound of $A+B$ is a
  dominant and isolated spectral value of $A+B$, an algebraically simple
  eigenvalue and a first order pole of the resolvent. Moreover,
  the resolvent $\Res(\phdot,A+B)$ is uniformly eventually strongly
  positive at $\spb(A+B)$ with respect to $u$. We now see from
  \cite[Theorem~4.2]{DanersNOTE} that the spectral projection $P$
  associated with $\spb(A+B)$ is strongly positive with respect to $u$.
  
  Next we observe that assumptions (i) and (ii) imply that
  $(e^{t(A+B)})_{t \geq 0}$ is eventually (in fact: immediately) norm
  continuous and that $e^{tA}E \subseteq E_u$ for all $t > 0$. Indeed, if
  (i) is fulfilled, then it follows from
  \cite[Proposition~III.1.12(i)]{Engel2000} that the perturbed semigroup
  $(e^{t(A+B)})_{t \geq 0}$ is analytic, too. From $D(A+B) = D(A) \subseteq
  E_u$ we can thus conclude that $e^{t(A+B)}E \subseteq D(A+B) \subseteq
  E_u$ for every $t > 0$. If, on the other hand, (ii) is fulfilled, then
  $(e^{tA})_{t \geq 0}$ is immediately norm continuous according to
  \cite[Theorem~III.1.16(i)]{Engel2000}.  Moreover, we obviously have
  $e^{tA}E \subseteq E = E_u$.
  
  Let us now show that the two properties proved above imply that
  $(e^{t(A+B)})_{t \geq 0}$ is individually eventually strongly positive
  with respect to $u$. Since the perturbed semigroup is eventually norm
  continuous and the spectral bound $\spb(A+B)$ is a dominant spectral value
  of $A+B$ and a first order pole of its resolvent, it
  follows that the rescaled semigroup $(e^{t(A+B - \spb(A+B)I)})_{t \ge
    0}$ is bounded. Since the spectral projection $P$ associated with
  $\spb(A+B)$ is strongly positive with respect to $u$, we conclude from
  the characterisation theorem given in \cite[Theorem~5.2]{Daners2016a}
  that $(e^{t(A+B)})_{t \geq 0}$ is individually eventually strongly
  positive with respect to $u$.
\end{proof}

A typical space where the condition $E_u = E$ in assumption~(ii) of the
above theorem is fulfilled is the space $\Cont(K;\bbC)$ of all
complex-valued continuous functions on a compact Hausdorff space $K$;
this holds independently of the choice of the quasi-interior point $u$.

\begin{examples}
  Let $\Omega \subseteq \bbR^d$ be a bounded domain of class $C^2$. 
  Consider one of the following situations:
  \begin{enumerate}[\upshape (a)]
    \item $E = C_0(\Omega;\bbC)$ (the space of all complex-valued continuous functions
    on $\Omega$ which vanish at the boundary) and $A = -\Delta_D^2$, where $\Delta_D$ denotes
    the Dirichlet Laplace operator on $E$.
    \item $E = C(\overline{\Omega};\bbC)$ and $A = -(\Delta_R^c)^2$, where
      $\Delta_R^c$ denotes the Laplace operator on $E$ with Robin boundary
      conditions (see \cite[Section~6.4]{Daners2016} for details).
  \end{enumerate}
  Then $E$ and $A$ fulfil the assumptions of 
  Theorem~\ref{thm:qualitative-perturbation-sg}. For~(a), this is shown
  in the proof of \cite[Theorem~6.1]{Daners2016a} and for~(b), this
  follows from \cite[Sections~6.3 and~6.4]{Daners2016}. 
\end{examples}

In case that the perturbation $B$ is compact, we can replace
assumption~(ii) in Theorem~\ref{thm:qualitative-perturbation-sg} with a
weaker condition. This is the subject of the next theorem, our second
main result in this section.

\begin{theorem}
  \label{thm:qualitative-perturbation-sg-compact}
  Suppose that the assumptions of
  Theorem~\ref{thm:qualitative-perturbation-sg} are fulfilled, but
  instead of {\upshape (i)} or {\upshape (ii)} assume the following
  condition:
  \begin{enumerate}
  \item[\upshape (iii)] $(e^{tA})_{t \geq 0}$ is eventually norm
    continuous and $E_u = E$.
  \end{enumerate}
  Then there is an $\varepsilon > 0$ such that for every compact
  operator $0 \leq B \in \calL(E)$ with $\|B\| < \varepsilon$ the
  semigroup generated by $A+B$ is individually eventually strongly
  positive with respect to $u$.
\end{theorem}
\begin{proof}
  The proof is the same as for
  Theorem~\ref{thm:qualitative-perturbation-sg}. The only difference is
  that we need the compactness of $B$ to conclude that the perturbed
  semigroup $(e^{t(A+B})_{t \geq 0}$ is eventually norm continuous since
  the original semigroup $(e^{tA})_{t \geq 0}$ is only assumed to be
  eventually but not necessarily immediately norm continuous; see
  \cite[Proposition~III.1.14]{Engel2000}
\end{proof}

Let us also comment on perturbation by (non-positive) multiplication
operators:

\begin{corollary}
  The Theorems~\ref{thm:qualitative-perturbation-sg}
  and~\ref{thm:qualitative-perturbation-sg-compact} remain true if we
  replace the assumption of $B$ being positive with the assumption that
  $B$ be a real multiplication operator (where, however, $\varepsilon$
  has to be chosen half as large as in the theorems)
\end{corollary}
\begin{proof}
  The operator $\tilde B := B + \|B\|\id$ has norm at most $2\|B\|$ and
  is positive according to
  Lemma~\ref{lem:norm-of-multiplication-operator}.  Hence, the corollary
  follows from Theorems~\ref{thm:qualitative-perturbation-sg}
  and~\ref{thm:qualitative-perturbation-sg-compact} and from the formula
  \begin{displaymath}
    e^{t(A+B)} = e^{-t\|B\|}e^{t(A+\tilde B)}
  \end{displaymath}
  which is true for all $t \geq 0$.
\end{proof}

We can prove a much stronger result than in the above theorems in case
that $E$ is finite dimensional.

\begin{proposition}
  \label{prop:open-generator-set}
  Let $d \in \bbN$, $d\geq 1$. The set of all generators of eventually
  strongly positive $C_0$-semigroups on $\bbC^d$ is an open subset of
  $\bbR^{d\times d}$.
\end{proposition}
\begin{proof}
  Let $A \in \bbC^{d \times d}$ be the generator of an eventually
  strongly positive semigroup. Then obviously, $A \in \bbR^{d \times
    d}$. By the characterisation result in
  \cite[Corollary~5.6]{Daners2016} this implies that $\spb(A)$ is a
  dominant spectral value of $A$ and that the corresponding spectral
  projection $P_0$ has only strictly positive entries. Hence, it follows
  from \cite[Proposition~3.1]{Daners2016} that $\spb(A)$ is an
  algebraically simple eigenvalue of $A$. We now conclude from
  Lemma~\ref{lem:spectral-bound-perturbation} that for all $B \in
  \bbR^{d \times d}$ which are sufficiently small in norm, the spectral
  bound $\spb(A+B)$ is a dominant spectral value of $A+B$.  Moreover,
  the spectral projection $P_B$ corresponding to $\spb(A+B)$ fulfils
  $P_B \to P_0$ as $\|B\| \to 0$. Since $P_B$ is real, it thus
  contains only strictly positive entries whenever $\|B\|$ is
  sufficiently small and thus, we can again employ the characterisation
  result in \cite[Corollary~5.6]{Daners2016} to conclude that the
  semigroup $(e^{t(A+B)})_{t \geq 0}$ is eventually strongly positive for
  all such $B$.
\end{proof}

It is a natural question whether the set of all generators of
strongly positive matrix semigroups $(e^{tA})_{t \geq 0}$ (meaning that
each matrix $e^{tA}$ has only strictly positive entries whenever $t >
0$) is also open in $\bbR^{d \times d}$.  Surprisingly, the answer
depends on the dimension $d$: it is positive if $d = 2$ (and, obviously,
also if $d = 1$), but negative if $d \geq 3$. The details can be found in
the next corollary and the subsequent example.

\begin{corollary}
  The set all generators of strongly positive $C_0$-semigroups on
  $\bbC^2$ is an open subset of $\bbR^{2\times 2}$
\end{corollary}
\begin{proof}
  The generator of a strongly positive $C_0$-semigroup on $\bbC^2$ is
  obviously a real matrix and it was shown in
  \cite[Proposition~6.2]{Daners2016} that a matrix $A \in \bbR^{2 \times
    2}$ generates a strongly positive $C_0$-semigroup if and only if it
  generates and eventually strongly positive $C_0$-semigroup. Hence, the
  corollary follows from Proposition~\ref{prop:open-generator-set}.
\end{proof}

\begin{example}
  \label{ex:perturbation-positive-sg}
  We showed in Proposition~\ref{prop:open-generator-set} that eventual
  strong positivity of matrix semigroups is robust with respect to
  small, not necessarily positive perturbations. We now give an example
  that this is not the case for strong positivity of the
  semigroup. Consider the generator
  \begin{displaymath}
    A:=\frac{1}{\sqrt{2}}
    \begin{bmatrix}
      0&0&1\\
      0&0&1\\
      1&1&0
    \end{bmatrix}
  \end{displaymath}
  on $\bbC^3$. Then it is easily checked that
  \begin{displaymath}
    A^{2k}=\frac{1}{2}
    \begin{bmatrix}
      1&1&0\\
      1&1&0\\
      0&0&2
    \end{bmatrix}
    \qquad\text{and}\qquad
    A^{2k+1}=A
  \end{displaymath}
  for all $k\geq 1$ and thus
  $e^{tA}=\sum_{k=0}^\infty\frac{t^k}{k!}A^k\gg 0$ for all $t>0$.
  Moreover, $\sigma(A)=\{0,\pm 1\}$, where the eigenspace associated
  with $1$ is spanned by the positive eigenvector $(1,1,\sqrt{2})$. If
  we set
  \begin{displaymath}
    B:=\frac{1}{\sqrt{2}}
    \begin{bmatrix}
      0&-1&0\\
      -1&0&0\\
      0&0&0
    \end{bmatrix},
  \end{displaymath}
  then
  \begin{displaymath}
    A+\varepsilon B=\frac{1}{\sqrt{2}}
    \begin{bmatrix}
      0&-\varepsilon&1\\
      -\varepsilon&0&1\\
      1&1&0
    \end{bmatrix}
  \end{displaymath}
  cannot generate a positive semigroup for any $\varepsilon>0$. However,
  if $\varepsilon$ is small enough, then
  Theorem~\ref{thm:qualitative-perturbation-sg} implies that
  $(e^{t(A+\varepsilon B})_{t\geq 0}$ is eventually strongly positive.
\end{example}
Clearly, the above example can be generalised to any finite dimension $d \ge 3$
by defining
\begin{displaymath}
  A:=\frac{1}{\sqrt{d-1}}
  \begin{bmatrix}
    0      & \dots &    0   & 1      \\
    \vdots &       & \vdots & \vdots \\
    0      & \dots &    0   & 1      \\
    1      & \dots &    1   & 0
  \end{bmatrix} \in \bbC^{d,d}.
\end{displaymath}
Hence, Proposition~\ref{prop:open-generator-set} and the above example 
show that, in dimension $d \geq 3$, eventual strong positivity of 
semigroups is a much more stable concept than strong positivity.

\subsection{A quantitative result}
In this final section we prove a quantitative perturbation theorem for
semigroups.  It is based on the quantitative result about resolvents in
Theorem~\ref{theorem:a-compact-perturbation}. It seems, in general,
unclear how to ensure that the real spectral bound $\spb_\bbR(A+B)$ (see
the discussion before Theorem~\ref{theorem:a-compact-perturbation} for a
definition) is a dominant spectral value of $A+B$ (and thus coincides
with $\spb(A+B)$). For this reason we restrict ourselves to self-adjoint
semigroups and perturbations on Hilbert spaces in the following theorem;
since the spectrum of self-adjoint operators is always real we clearly
have $\spb_\bbR(A+B) = \spb(A+B)$ in case that $A$ and $B$ are
self-adjoint.

\begin{theorem}
  \label{thm:quantitative-sg-perturbation-on-hilbert-space}
  Let $\{0\} \not= H$ be a complex-valued $L^2$-space over an arbitrary
  measure space, let $u \in H_+$ be a quasi-interior point and let
  $(e^{tA})_{t \geq 0}$ be a self-adjoint and real $C_0$-semigroup on $H$
  with $D(A) \subseteq H_u$. Suppose that there is a $\lambda_1 > \spb(A)$
  such that $\Res(\lambda,A) \gg_u 0$ for all $\lambda \in
  (\spb(A),\lambda_1)$.
  
  If $B \in \calL(H)$ is positive and self-adjoint with $\|B\| <
  \lambda_1-\spb(A)$, then the semigroup $(e^{t(A+B)})_{t \geq 0}$ is
  uniformly eventually strongly positive with respect to $u$.
\end{theorem}

Note that if the underlying measure space of $H$ is $\sigma$-finite, then a vector
$u \in H_+$ is a quasi-interior point if and only if $u(\omega) > 0$ for 
almost all $\omega$ in the measure space. The fact that we obtain \emph{uniform}
eventual strong positivity for the perturbed semigroup 
in the above theorem is due to a recent result of the authors in the Hilbert space 
case which appeared in the second author's PhD thesis
\cite[Theorem~10.2.1]{GlueckDISS}.

\begin{proof}[Proof of Theorem~\ref{thm:quantitative-sg-perturbation-on-hilbert-space}]
  Since $e^{t_0A}H \subseteq H_u$, it follows from
  \cite[Theorem~2.3(ii)]{DanersNOTE} that $e^{t_0A}$ is
  compact. Therefore, the semigroup $(e^{tA})_{t\geq 0}$ is eventually
  compact, and since it is analytic, it must in fact be immediately
  compact \cite[Exercise~II.4.30(6)]{Engel2000}. Hence, its generator
  $A$ has compact resolvent \cite[Theorem~II.4.29]{Engel2000}. In
  particular, $\spb(A)$ is a pole of $\Res(\phdot,A)$ and $B$ is
  $A$-compact.
  
  We have $M := \sup_{\repart \lambda \geq \lambda_1} \|\Res(\lambda,A)\| =
  \frac{1}{\lambda_1-\spb(A)}$ since $A$ is self-adjoint. Moreover,
  $\spb_\bbR(A+B)$ equals $\spb(A+B)$ since $A+B$ is self-adjoint. It
  therefore follows from Theorem~\ref{theorem:a-compact-perturbation}
  that $\spb(A+B)$ is a pole of $\Res(\phdot,A+B)$ and that $\Res(\phdot,A+B)$
  is uniformly eventually strongly positive at $\spb(A+B)$ with respect
  to $u$. Hence, the spectral projection $P$ associated with the
  spectral value $\spb(A+B)$ of $A+B$ fulfils $P \gg_u 0$ according to
  \cite[Theorem~4.1]{DanersNOTE}. Since $D(A+B) = D(A) \subseteq H_u$ it
  thus follows from the characterisation of eventual positivity in
  \cite[Theorem~10.2.1]{GlueckDISS} that $(e^{t(A+B)})_{t \geq 0}$ is
  uniformly eventually strongly positive with respect to $u$.
\end{proof}

Let us conclude the paper with the following example of a Laplace
operator on $(0,1)$ with non-local boundary conditions. Eventual
positivity properties of the unperturbed operator were discussed in
\cite[Section~6]{Daners2016a} and in \cite[Section~11.7]{GlueckDISS}.

\begin{example}
  Let $H = L^2(0,1)$ and consider the sesqui-linear form $a\colon
  H^1(0,1) \times H^1(0,1) \to \bbC$ which is given by
  \begin{displaymath}
    a(u,v) = \int_0^1 u' \overline{v}' dx +
    \begin{bmatrix}
      u(0) & u(1)
    \end{bmatrix}
    \begin{bmatrix}
      1 & 1 \\
      1 & 1
    \end{bmatrix}
    \begin{bmatrix}
      v(0)\\v(1)
    \end{bmatrix}.
  \end{displaymath}
  The operator $A$ on $H$ associated with $a$ is self-adjoint; it is a
  (negative) Laplace operator with non-local boundary conditions, given
  by
  \begin{displaymath}
    \begin{aligned}
      D(A) & = \bigl\{u \in H^2(0,1)
      \colon u'(0) = -u'(1) = u(0) + u(1)\bigr\}, \\
      A u & = -u''.
    \end{aligned}
  \end{displaymath}
  It was shown in \cite[Theorem~11.7.3]{GlueckDISS} that $\spb(-A) < 0$
  and that the semigroup $(e^{-tA})_{t \geq 0}$ is not positive, but
  uniformly eventually strongly positive with respect to $\one$ 
  (where $\one$ denotes the constant function on $(0,1)$ with value $1$).
  Let us now prove the following assertion:
  
  \emph{If $0 \leq B \in \calL(H)$ is self-adjoint and $\|B\| < 1$, then
    the semigroup generated by $-A+B$ is uniformly eventually
    strongly positive with respect to $\one$.}
  
  \begin{proof}
    Obviously, $u = \one$ is a quasi-interior point of $H_+$. Moreover,
    we have $D(A) \subseteq H^1(0,1) \subseteq L^\infty(0,1) = H_{\one}$. It was shown in
    \cite[Theorem~6.11(i)]{Daners2016a} that $\Res(0,-A) \gg_u 0$ and in
    the proof of this theorem the following formula for $\Res(0,-A)$ was
    given:
    \begin{displaymath}
      (\Res(0,-A)f)(x)
      = \frac{1}{2} \int_0^x \int_y^1  f(z) \, dz \, dy 
      + \frac{1}{2} \int_x^1 \int_0^y f(z) \,  dz \, dy
    \end{displaymath}
    for all $f \in H$ and all $x \in [0,1]$. We want to apply
    Theorem~\ref{thm:quantitative-sg-perturbation-on-hilbert-space} and
    thus we have to estimate the number $-\spb(-A)$. This number
    coincides with the distance of $0$ to $\sigma(-A)$ which is in turn
    equal to $\frac{1}{\|\Res(0,-A)\|}$ since $-A$ is self-adjoint. Hence,
    we have to estimate the norm of $\Res(0,-A)$. Since this is a positive
    operator, we only have to consider positive functions $f \in H$. For
    each such $f$ we have
    \begin{displaymath}
      \|\Res(0,-A)f\|_2 
      \leq \|\Res(0,-A)f\|_\infty 
      \leq \int_0^1 \int_0^1 f(z) \, dz \, dy
      = \|f\|_1 \leq \|f\|_2.
    \end{displaymath}
    Hence, we have $\|\Res(0,-A)\| \leq 1$ and thus $\frac{1}{\|\Res(0,-A)\|}
    \geq 1$. The assertion now follows from
    Theorem~\ref{thm:quantitative-sg-perturbation-on-hilbert-space}.
  \end{proof}
\end{example}


\section{Appendix: Formulas for rank-$1$-perturbations}
\label{appendix:formulas}

Let $A\colon E \supseteq D(A) \to E$ be a linear operator on a complex
Banach space $E$. In this appendix we study what happens to the spectrum
and the resolvent of $A$ if we perturb $A$ by a rank-$1$-operator. If
$A$ generates a $C_0$-semigroup and if the perturbation is, in a sense,
well-adapted to $A$, we also derive a formula for the perturbed
$C_0$-semigroup.

\begin{proposition}
  \label{prop:rank-1-pert-of-resolvent}
  Let $A\colon E \supseteq D(A) \to E$ be a linear
  operator on a complex Banach space $E$ and let $\lambda \in
  \rho(A)$. Moreover, assume that $\varphi \in E'$ and $w \in E$.
  
  Then $\lambda \in \rho(A + \varphi\otimes w)$ if and only if $1 \neq
  \langle \varphi, \Res(\lambda,A) w \rangle$. In this case we have
  \begin{equation}
    \label{eq:resolvent-rank-one-pert}
    \Res(\lambda, A + \varphi \otimes w)
    = \Res(\lambda,A) + \frac{1}{1-\langle\varphi,\Res(\lambda,A)w\rangle}
    \Res(\lambda,A) \, (\varphi\otimes w) \, \Res(\lambda,A).
  \end{equation}
\end{proposition}
\begin{proof}
  If $\langle \varphi, \Res(\lambda,A) w \rangle\neq 1$, then the right
  hand side of \eqref{eq:resolvent-rank-one-pert} is well defined and,
  using that $(\varphi \otimes w) \, \Res(\lambda,A) \, (\varphi \otimes w) =
  \langle \varphi, \Res(\lambda,A)w \rangle (\varphi \otimes w)$, one can
  check by a simple computation that it is the inverse of $\lambda I -
  (A + \varphi \otimes w)$; this implies that $\lambda \in \rho(A +
  \varphi \otimes w)$ and that the formula holds.
  
  Now, assume that $\lambda \in \rho(A + \varphi \otimes w)$. If
  $w=0$, then clearly $\langle \varphi, \Res(\lambda,A) w
  \rangle=0\neq 1$, so let $w \not= 0$. Since both operators 
  $\lambda I - A$ and
  \begin{displaymath}
    \lambda I - (A + \varphi \otimes w) = \bigl(I - (\varphi \otimes w) 
    \Res(\lambda,A)\bigr) (\lambda I - A)
  \end{displaymath}
  are bijective from $D(A)$ to $E$, it follows that
  $I - (\varphi \otimes w) \Res(\lambda,A)$ is a bijection on $E$; 
  in particular, the latter operator is injective, so
  \begin{displaymath}
  	0 \not= \bigl(I - (\varphi \otimes w) \Res(\lambda,A)\bigr)w =
  	\bigl(1 - \langle \varphi, \Res(\lambda,A)w \rangle\bigr)w.
  \end{displaymath}
  This proves that $\langle \varphi, \Res(\lambda,A)w \rangle \not= 1$.
\end{proof}

If $A$ is a square matrix and $\lambda = 0$, then
\eqref{eq:resolvent-rank-one-pert} is a special case of the
\emph{Sherman--Morrison--Woodbury formula} from numerical analysis; see
\cite[Section~2.1.3]{Golub1989} or \cite[Lemma on p.~68]{Miller1981}.

If $A$ generates a $C_0$-semigroup $(e^{tA})_{t \geq 0}$, then we do not
obtain such a nice formula for the perturbed semigroup $(e^{t(A +
  \varphi \otimes w)})_{t \geq 0}$, in general. However, if $w$ is an
eigenvector of $A$, then an explicit formula for the perturbed semigroup
can be given, and the perturbation formula for the resolvent from the
previous proposition can be considerably simplified.

\begin{proposition}
  \label{prop:rank-1-pert-special-case}
  Let $A\colon E \supseteq D(A) \to E$ be a linear
  operator on a complex Banach space $E$. Let $\varphi \in E'$ and let
  $v \in D(A)$ be an eigenvector of $A$ for an eigenvalue $\lambda_0 \in
  \bbC$.
  \begin{enumerate}[\upshape (a)]
  \item If $\lambda \in \rho(A)$, then $\lambda \in \rho(A + \varphi
    \otimes v)$ if and only if $\lambda - \lambda_0 \not= \langle
    \varphi, v \rangle$. In this case
    \begin{equation}
      \label{eq:resolvent-rank-one-pert-special-case}
      \Res(\lambda, A + \varphi \otimes v)
      = \Res(\lambda,A)+\frac{1}{(\lambda-\lambda_0)-\langle\varphi,v\rangle}
      \bigl( \varphi \otimes v \bigr)\Res(\lambda, A).
    \end{equation}
  \item If $A$ generates a $C_0$-semigroup on $E$ and if
    $\langle\varphi,v\rangle + \lambda_0 \not\in \sigma(A)$, then
    \begin{equation}
      \label{eq:semigroup-rank-one-pert-special-case}
      e^{t(A+ \varphi \otimes v)}
      = e^{tA}+(\varphi\otimes v)
      \bigl(e^{t(\langle\varphi,v\rangle+\lambda_0)}I-e^{tA}\bigr)
      \Res(\langle\varphi,v\rangle+\lambda_0,A)
    \end{equation}
    for all $t \geq 0$.
  \end{enumerate}
\end{proposition}
\begin{proof}
  (a) This follows immediately from
  Proposition~\ref{prop:rank-1-pert-of-resolvent}.
  
  (b) The right hand side of
  \eqref{eq:semigroup-rank-one-pert-special-case} is clearly strongly
  continuous with respect to $t \in [0,\infty)$ and a direct computation
  verifies that it is a semigroup. We denote by $B$ the generator of
  this semigroup. Then one immediately checks that $D(B) \supseteq D(A) =
  D(A +\varphi \otimes v)$ and that $Bf = (A + \varphi \otimes v)f$ for
  all $f \in D(A) = D(A +\varphi \otimes v)$.  Hence, $B$ is an
  extension of $A +\varphi \otimes v$.  Since $A +\varphi \otimes v$ and
  $B$ are both semigroup generators, their resolvent sets have non-empty
  intersection and thus, we must have $B = A + {\varphi \otimes v}$.
\end{proof}

\pdfbookmark[1]{\refname}{biblio}%
\bibliographystyle{doi}%
\bibliography{literature_evpos}

\end{document}